\documentclass[final]{siamltex}

\usepackage{amsfonts}
\usepackage{graphicx}
\usepackage{subfigure}

\usepackage{amsmath}
\usepackage{amssymb}
\usepackage{epsf}

\usepackage{color}

\newtheorem{prop}{Proposition}

\begin{document}

\thispagestyle{empty}
\bibliographystyle{siam}

\title{Ranking Hubs and Authorities using\\
 Matrix Functions}

\author {
Michele Benzi\thanks{
Department of Mathematics and Computer
Science, Emory University, Atlanta, Georgia 30322, USA
(benzi@mathcs.emory.edu).
The work of this author was supported
by National Science Foundation grant
DMS1115692.} \and
Ernesto Estrada\thanks{Department of Mathematics and Statistics
and Institute of Complex Systems,
University of Strathclyde, Glasgow G1 1XQ, UK 
(ernesto.estrada@strath.ac.uk). The work 
of this author was supported in part by the
New Professor's Fund, University of Strathclyde, 
and by the project \lq\lq Mathematics of Large Technological 
Evolving Networks (MOLTEN)\rq\rq\ funded by the 
Engineering and Physical Sciences Research Council 
and by the Research Councils UK Digital Economy programme, 
grant ref.~EP/I016058/1.
} \and
Christine Klymko\thanks{Department of Mathematics and Computer Science,
Emory University, Atlanta, Georgia 30322, USA
(cklymko@emory.edu). The work of this author was supported in part
by the Laney Graduate School of Arts and Sciences at Emory University
and by the National Science Foundation.}
}

\maketitle

\markboth{{\sc M.~Benzi, E.~Estrada, and C.~Klymko}}{{Ranking hubs and
authorities using matrix functions}}

\begin{abstract}
The notions of subgraph centrality and 
communicability, based on the exponential of the adjacency
matrix of the underlying graph, have been effectively used in the analysis
of undirected networks. In this paper we propose an extension
of these measures to directed networks, and we apply them
to the problem of ranking hubs and authorities.
  The extension is achieved by {\em bipartization}, i.e.,
the directed network is mapped onto a bipartite undirected 
network with twice as many nodes in order to obtain a network 
with a symmetric adjacency matrix.  
We explicitly determine the exponential of this adjacency matrix
in terms of the adjacency matrix of the original, directed network,
and we give an interpretation of centrality and communicability 
in this new context, leading to a technique 
for ranking hubs and authorities. 
 The matrix exponential method for computing hubs and authorities 
is compared to the well known HITS algorithm, both on small artificial
examples 
and on more realistic real-world networks. A few other ranking algorithms
are also discussed and compared with our technique. 
The use of Gaussian quadrature rules for calculating hub and
authority scores is discussed.
\end{abstract}



\section{Introduction}
In recent years, the study of networks has become central to many 
disciplines \cite{Bocca06,Brandes,Caldarelli,Ebook11,Ebook10,New03,New10,NewBarWat03}.
 Networks can be used to describe and analyze many different types of interactions, from those between people (social networks), to the flow of goods across an area (transportation networks), to links between websites (the WWW graph), and
so forth.  In general, a network is a set of objects (nodes) and the connections between them (edges).  Often, research is focused on determining and describing important structural characteristics of a network or the interactions among 
its components. 

One common question in network analysis is to determine the most 
``important" nodes (or edges) in the network, also called 
{\em node} or {\em vertex} ({\em edge}) {\em centrality}.  
The interpretation of what is meant by ``important" can change 
from application to application. 
Due to this, many different measures of centrality have been developed.  
For an overview, see \cite{Brandes}. A closely related notion is that of {\em rank}
of a node in a network. There exist a number of definitions and algorithms
for computing rankings; see, e.g., 
\cite{franceschet,IRSurvey,KG11,Pagerank,Whos,vigna} for up-to-date overviews. 

The main notion of node centrality considered in this paper,
{\em subgraph centrality}, 
was introduced by Estrada and 
Rodr\'iguez-Vel\'azquez
in \cite{estradarodriguez05}. We refer readers to 
\cite{estradarodriguez05} for the
motivation behind this notion and for its name; see also the review article
\cite{NetworkProp}, and the discussion in section \ref{SC}.  
The interpretation of centrality described in \cite{NetworkProp} 
applies mostly to 
undirected networks. However, many important real-world networks 
(the World Wide Web, the Internet, citation networks, food webs, 
certain social networks, etc.)~are
directed. One goal of this paper is to extend the notions of centrality 
and communicability described in 
\cite{estradahatano2,NetworkProp} to directed networks, 
with an eye towards developing new %hubs and authorities 
ranking algorithms for, e.g., document collections, web pages, and so forth.
We further compare our approach with some standard algorithms, such as HITS
(see \cite{HITS}) and a few others. 
Methods of quickly determining hub and authority 
rankings using Gauss-type quadrature rules are also 
discussed.

\section{Basic notions}
\label{sec:digraphs}
Here we briefly review some basic graph-theoretic notions; 
we refer to \cite{Die00} for a
comprehensive treatment.
A \emph{graph} $G=(V,E)$ is formed by a set of nodes (vertices) 
$V$ and edges $E$ formed by unordered pairs of vertices.  
Every network is naturally associated with a graph $G=(V,E)$ 
where $|V|$ is the number of nodes in the network and $E$ is the 
collection of edges between objects,  
$E=\{(i,j) \, | {\rm \, there \,\, is \,\, an \,\, edge \,\, between \,\, 
node \,}\, i\, {\rm \, and \,} \, {\rm node\,}\,   j  \}$.  
The {\em degree} $d_i$
of a vertex $i$ is the number of edges incident to $i$.

A directed graph, or \emph{digraph} $G=(V,E)$ is formed by a 
set of vertices $V$ and 
edges $E$ formed by ordered pairs of vertices.  That is, 
$(i,j) \in E \not\Rightarrow (j,i) \in E$.  In the case of digraphs, 
which model directed networks, there are two types of degree.  
The {\em in-degree} of node $i$ is given by the number of edges 
which point to $i$.  The {\em out-degree} is given by the number 
of edges pointing out from $i$.

A \emph{walk} is a sequence of vertices $v_1, v_2, \ldots, v_k$ 
such that for $1\leq i < k$, there is an edge between $v_i$ and 
$v_{i+1}$ (or a directed edge from $v_i$ to $v_{i+1}$ in the 
case of a digraph).  Vertices and edges may be repeated. 
A walk is \emph{closed} if $v_1=v_k$.
A {\em path} is a walk consisting only of distinct vertices. 

A graph $G$ is {\em connected} if every pair of vertices is linked
by a path in $G$.
A digraph is {\em strongly connected} if for any pair of vertices 
$v_i$ and $v_k$ there is a walk starting at $v_i$ and ending at $v_k$.
A digraph is {\em weakly connected} if the graph obtained by 
disregarding the orientation of its edges is connected.
Unless otherwise specified, every digraph in this paper is simple
(unweighted with no multiple edges or loops and connected). Note,
however, that most
of the techniques and results in the paper can be extended without
difficulty to more general digraphs, in particular weighted ones.

The \emph{adjacency matrix} of a graph is a matrix $A \in \mathbb{R}^{|V| \times |V|}$ defined in the following way:
$$A = (a_{ij}); \quad a_{ij}=\left\{\begin{array}{ll}
1,& \textnormal{ if } (i,j) \textnormal{ is an edge in } G,\\
0, & \textnormal{ else. }
\end{array}\right .
$$

Under the conditions imposed on $G$, $A$ has zeros on the diagonal.  If $G$ is an 
undirected graph, $A$ will be a symmetric matrix and the eigenvalues will be real.  In the case of digraphs, $A$ is not symmetric and may have complex (non-real)
eigenvalues. 

\section{Kleinberg's HITS algorithm}
Here we briefly recall the classical {\em Hypertext Induced Topics Search} (HITS)
algorithm, first introduced by J.~Kleinberg in \cite{HITS}.
This algorithm provides the motivation for the extension of subgraph
centrality to directed graphs given in section \ref{ExtDig}.

\subsection{The basic iteration}
The HITS algorithm is based on the idea that in the World Wide Web, and indeed
in all document collections which can be represented by directed networks, 
there are two types of important nodes: {\em hubs} and {\em authorities}.  
Hubs are nodes which point to many nodes of the type considered important.  Authorities are these important nodes.  From this comes a circular definition: good hubs are those which point to many good authorities and good authorities are those pointed to by many good hubs.

Thus, the HITS ranking relies on an iterative method converging to a stationary solution.  Each node $i$ in the network is assigned two non-negative weights, an {\em authority weight} $x_i$ and a {\em hub weight} $y_i$.  To begin with, each $x_i$ and $y_i$ is given
an arbitrary nonzero value.  Then, the weights are updated in the following ways:
\begin{equation}
x^{(k)}_i = \sum_{j:(j,i) \in E} y^{(k-1)}_j {\,\,\,\,\,\ \rm and \,\,\,\,\,\,} y^{(k)}_i = \sum_{j:(i,j) \in E} x^{(k)}_j {\rm \,\,\,\,\, for \,\,\,\,\,} k=1,2,3...
\end{equation}
The weights are then normalized so that $\sum_j (x_j^{(k)})^2 =1$ and 
$\sum_j (y_j^{(k)})^2 =1$.

The above iterations occur sequentially and it can be shown that, under
mild conditions, both sequences of 
vectors $\{x^{(k)}\}$ and $\{y^{(k)}\}$ converge as $k \to \infty$.  
In practice, the iterative process is continued 
until there is
no significant change between consecutive iterates.

This iteration sequence shows the natural dependence relationship between hubs and authorities: if a node $i$ points to many nodes with large $x$-values, it receives a large $y$-value and, if it is pointed to by many nodes with large $y$-values, it receives a large $x$-value.

In terms of matrices, the equation (3.1) becomes: 
$x^{(k)} = A^Ty^{(k-1)}$ and $y^{(k)} = Ax^{(k)}$, followed by
normalization in the 2-norm.
This iterative process can be expressed as 
\begin{equation}
x^{(k)} = c_kA^TAx^{(k-1)} {\rm \,\,\,\,\, and \,\,\,\,\,} y^{(k)} = c_k'AA^Ty^{(k-1)},
\end{equation}
where $c_k$ and $c_k'$ are normalization factors.
A typical choice for the inizialization vectors $x^{(0)}$, $y^{(0)}$ 
would be the constant vector
$$x^{(0)} = y^{(0)} = [1/\sqrt{n},\, 1/\sqrt{n},\, \ldots ,1/\sqrt{n}],$$ 
see \cite{EIHITS}.
Hence, HITS is just an iterative power method to 
compute the dominant eigenvector for $A^TA$ and for $AA^T$.  
The authority scores are determined by the 
entries of the dominat eigenvector of the matrix $A^TA$, 
which is called the {\em authority matrix} and the hub 
scores are determined by the entries of the
dominant eigenvector of $AA^T$, called the {\em hub matrix}. Recall that 
the eigenvalues of  both $A^TA$ and $AA^T$ are the squares of the 
singular values of $A$. Also, the eigenvectors of $A^TA$ are the right singular
vectors of $A$, and the eigenvectors of $AA^T$ are the left singular 
vectors of $A$.

\subsection{HITS reformulation}
\label{sec:newHITS}
In a digraph the adjacency matrix $A$ is generally nonsymmetric, 
however, the two matrices used in the HITS algorithm ($A^TA$ and $AA^T$) are symmetric.  Note that, setting
$$ \mathcal{A} =  \left( \begin{array}{cc}
0 & A \\
A^T & 0  \end{array} \right)\,,$$
a symmetric matrix is obtained.  Now, 
$$\mathcal{A}^2 =   \left( \begin{array}{cc} AA^T & 0 \\ 0 & A^TA  \end{array} \right); {\rm \,\,\,\,\,\,\,} \mathcal{A}^3 =   \left( \begin{array}{cc} 0 & AA^TA \\ A^TAA^T & 0  \end{array} \right). $$
In general, 
$$\mathcal{A}^{2k} =   \left( \begin{array}{cc} (AA^T)^k & 0 \\ 0 & (A^TA)^k  \end{array} \right); {\rm \,\,\,\,\,\,\,} \mathcal{A}^{2k+1} =   \left( \begin{array}{cc} 0 & A(A^TA)^k \\ (A^TA)^kA^T & 0  \end{array} \right). $$

Applying HITS to this matrix $\mathcal{A}$, $\mathcal{A}^T = \mathcal{A}$ 
so $\mathcal{A}^T\mathcal{A} = \mathcal{A}\mathcal{A}^T = \mathcal{A}^2$ 
and introducing the vector $u^{(k)} = \left( \begin{array}{c} y^{(k)}\\x^{(k)}
\end{array} \right)$
for $k=1,2,3,\ldots$, equation (3.2) becomes
\begin{equation}
u^{(k)} = \mathcal{A}^2 u^{(k-1)} =  \left( \begin{array}{cc} AA^T & 0 \\ 0 & A^TA  
\end{array} \right) u^{(k-1)},
\end{equation}
followed by normalization of the two vector components of $u^{(k)}$ so that 
each has 2-norm equal to 1.
Now, if $A$ is an $n \times n$ matrix, $\mathcal{A}$ 
is $2n \times 2n$ and vector $u^{(k)}$ is in 
$\mathbb{R}^{2n}$.  The first $n$ entries of $u^{(k)}$ 
correspond to the hub rankings of the nodes, while 
the last $n$ entries give the authority rankings.
Under suitable assumptions (see the discussion in \cite[Chapter 11.3]{Pagerank}),
as $k\to \infty$ the
sequence $\{u^{(k)}\}$ converges to the dominant nonnegative
eigenvector of $\cal A$, which yields the desired hub and
authority rankings. 

Hence, in HITS only information obtained from the dominant
eigenvector of $\cal A$ is used. It is natural to expect
that taking into account spectral information corresponding
to the remaining eigenvalues and eigenvectors of $\cal A$ 
may lead to improved results.

Among the limitations of HITS, we mention the possible dependence
of the rankings on the choice of the initial vectors $x^{(0)}$,
$y^{(0)}$, see \cite{EIHITS} for examples of this, and the 
fact that HITS hub/authority rankings tend to be ``degree-biased",
i.e., they are strongly correlated with
the out-/in-degrees of the corresponding nodes \cite{Dingetal}. The
latter property is in fact shared by most eigenvector-based rankings;
for a discussion of this phenomenon in the case of scale-free graphs,
see \cite{MiPapa}.

\section{Subgraph centralities and communicabilities}\label{SC}
In \cite{NetworkProp}, the authors review several 
measures to rank the nodes in an undirected network $A$ based on
the use of matrix functions, such as the matrix exponential $e^A$.
The {\em subgraph centrality} \cite{estradarodriguez05} of node $i$ is given by 
$[e^A]_{ii}$ and the {\em communicability} \cite{estradahatano2} between nodes 
$i$ and $j$ ($i \neq j$) is given by $[e^A]_{ij}$.
Nodes $i$ corresponding to higher values of
$[e^A]_{ii}$ are considered more important than nodes
corresponding to lower values.
Large values of $[e^A]_{ij}$ indicate that information flows
more easily between nodes $i$ and $j$ than between pairs
of nodes corresponding to lower values of the same quantity.
The {\em Estrada index} of the 
graph is given by ${\rm Tr}\,(e^A)=\sum_{i=1}^n[e^A]_{ii}$. 
This index, which provides a global
characterization of a network, is analogous to the {\em partition
function} in statistical mechanics and plays an important role in the
study of networks at the {\em macroscopic} level: 
quantities such as the
{\em natural connectivity}, the {\em total energy}, the {\em Helmholtz 
free energy} and the {\em entropy} of a network can all be expressed
in terms of the Estrada index \cite{EHB11}. 

Consider the power series expansion of $e^A$,
\begin{equation}
e^A = I + A + \frac{A^2}{2!} + \frac{A^3}{3!} + \cdots + \frac{A^k}{k!} + \cdots
\end{equation}
From graph theory, it is well known that if $A$ is the adjacency matrix of an undirected graph, $[A^k]_{ij} = [A^k]_{ji}$ counts the number of walks of length $k$ between nodes $i$ and $j$.  Thus, the subgraph centrality of node $i$, $[e^A]_{ii}$, counts the total number of closed walks starting at node $i$, penalizing longer walks by scaling walks of length $k$ by the factor $\frac{1}{k!}$.  The communicability between nodes $i$ and $j$, $[e^A]_{ij}$, counts the number of walks between nodes $i$ and $j$, again scaling walks of length $k$ by a factor of $\frac{1}{k!}$.

It is worth mentioning that normalization of  
the diagonal entries of $e^A$ by ${\rm Tr}\,(e^A)$ yields a
probability distribution on the nodes of the network, which
can be given the following interpretation: the $i$th diagonal entry
of $e^A/{\rm Tr}\,(e^A)$ is the probability of selecting any weighted 
self-returning (closed) walk that starts and ends at node $i$ among all  
the weighted self-returning walks that start at any node and return
to the same node. % CHEK THIS! 
The weights used (factorial penalization) ensure that the shortest 
walks receive more weight than the longer ones: hence, the subgraph
centrality of node $i$ is proportional to the
probability of finding a random walker walking \lq\lq nearby\rq\rq\ 
node $i$. 

Although the matrix exponential is certainly well-defined for any
matrix, whether symmetric or not, the {\em interpretation} of the   
notion of subgraph centrality 
for directed networks can be problematic. 
To see this, consider the directed path graph consisting of $n$ nodes,
with edge set $E = \{(1,2)\,, (2,3)\,, \ldots , (n-1,n)\}$ and
adjacency matrix
%\begin{displaymath}\label{AP}
\begin{equation}\label{AP}
A=\left(\begin{array}{ccccc}
0 & 1 & 0 &\cdots & 0\\
0 & 0 & 1 & \cdots & 0\\
\vdots & \ddots & \ddots & \ddots &\vdots \\
0 & 0 & 0 & \cdots & 1\\
0 & 0 & 0 & \cdots & 0 \end{array}\right) \,.
%\end{displaymath}
\end{equation}
The entries of $e^A$ are given by
$$[e^A]_{ij}=\left\{\begin{array}{ll}
1/(j-i)!,& \textnormal{ if } j\ge i,\\
0, & \textnormal{ else. }
\end{array}\right .
$$
In particular, the diagonal entries of $e^A$ are all equal to 1. 
Therefore, it is impossible to distinguish any of the nodes from
the others on the basis of this centrality measure; yet, it
is clear that the first and last node are rather special, and
certainly more \lq\lq peripheral\rq\rq\ (less \lq\lq central\rq\rq)   
than the other nodes.
Also, we note that the probabilistic interpretation given above
for undirected graphs is no longer meaningful for this example. 
Part of the
problem, of course, is that the path digraph contains no closed
walks. In the next section we show one way to extend the notion
of subgraph centrality to digraphs that is immune from such 
shortcomings, and correctly differentiates between nodes in the example
above.
(On the other hand, it is interesting to note that the interpretation
of the off-diagonal
entries of $e^A$ in terms of communicabilities is straightforward
for the directed path. All entries of $e^A$ below the
main diagonal are zero, reflecting the fact that information can only flow
from a node to higher-numbered nodes. Also, the entries of $e^A$ decay
rapidly away from the main diagonal, reflecting 
the fact that the ``ease" of communication 
between a node and a higher numbered one decreases rapidly with the distance.)

Another issue when extending the notions of subgraph centrality and
communicability to directed graphs is that
computational difficulties may arise. While the
computations involved do not pose a problem for small 
networks, many real-world networks are large enough that 
directly computing the exponential of the adjacency matrix 
is prohibitive.  In \cite{BenziBoito}, 
techniques for bounding and estimating individual entries of the 
matrix exponential using Gaussian quadrature rules are discussed; see also
\cite{bonchietal} and section \ref{sec:approx} below.  
The ability to find upper and lower bounds for the 
entries requires that the matrix be symmetric, 
thus these bounds cannot be directly computed using the 
adjacency matrix of a directed network.
Again, these difficulties can be
circumvented using the approach discussed in the 
next section.

\section{An extension to digraphs}\label{ExtDig}
Although the techniques described in 
\cite{BenziBoito} cannot be 
directly applied to non-symmetric matrices, setting
\begin{equation}\label{calA}
\mathcal{A} =  \left( \begin{array}{cc}
0 & A \\
A^T & 0  \end{array} \right)
\end{equation}
produces a symmetric matrix $\mathcal{A}$ and, thus, upper 
and lower bounds of individual entries of $e^{\mathcal{A}}$ 
can be computed.  In Proposition \ref{prop1} 
below we relate $e^{\mathcal{A}}$ to the underlying hub 
and authority structure of 
the original digraph. By $B^{\dagger}$ we denote the Moore--Penrose 
generalized inverse of 
matrix $B$.\\

\begin{prop}\label{prop1}
Let $\mathcal{A}$ be as described in equation (\ref{calA}).  Then,
$$e^{\mathcal{A}} =  \left( \begin{array}{cc}
\cosh{\left(\sqrt{AA^T}\right)} & A\left(\sqrt{A^TA }\right)^{\dagger}\sinh{\left(\sqrt{A^TA }\right)} \\
\sinh{\left(\sqrt{A^TA}\right)}\left(\sqrt{A^TA }\right)^{\dagger}A^T & \cosh{\left(\sqrt{A^TA}\right)}  \end{array} \right).$$
\end{prop}

\begin{proof}
Let $A = U\Sigma V^T$ be the SVD of the original (non-symmetric) adjacency matrix $A$.  Then, $\mathcal{A}$ can be decomposed as
 $\mathcal{A} = \left( \begin{array}{cc} U & 0 \\ 0 & V  \end{array} \right) \left( \begin{array}{cc} 0 & \Sigma \\ \Sigma & 0  \end{array} \right) \left( \begin{array}{cc} U^T & 0 \\ 0 & V^T  \end{array} \right)$.  Hence, 
\begin{equation} \label{expA}
e^{\mathcal{A}} = \left( \begin{array}{cc} U & 0 \\ 0 & V  \end{array}\right) \exp{\left( \begin{array}{cc} 0 & \Sigma \\ \Sigma & 0  \end{array} \right)} \left( \begin{array}{cc} U^T & 0 \\ 0 & V^T  \end{array} \right).
\end{equation}
Now,
$$\exp{\left( \begin{array}{cc} 0 & \Sigma \\ \Sigma & 0  \end{array} \right)} = \cosh{\left( \begin{array}{cc} 0 & \Sigma \\ \Sigma & 0  \end{array} \right)} + \sinh{\left( \begin{array}{cc} 0 & \Sigma \\ \Sigma & 0  \end{array} \right)}$$ $$= {\left( \begin{array}{cc} \cosh(\Sigma) & 0 \\ 0 & \cosh(\Sigma)  \end{array} \right)} + {\left( \begin{array}{cc} 0 & \sinh(\Sigma) \\ \sinh(\Sigma) & 0  \end{array} \right)}. $$
Thus,
\begin{equation} \label{expS}
\exp{\left( \begin{array}{cc} 0 & \Sigma \\ \Sigma & 0  \end{array} \right)} = {\left( \begin{array}{cc} \cosh(\Sigma) & \sinh(\Sigma) \\ \sinh(\Sigma) & \cosh(\Sigma) \end{array} \right)}.
 \end{equation}
Putting together equations (\ref{expA}) and (\ref{expS}),
$$e^{\mathcal{A}} =   \left( \begin{array}{cc} U & 0 \\ 0 & V  \end{array}\right)  {\left( \begin{array}{cc} \cosh(\Sigma) & \sinh(\Sigma) \\ \sinh(\Sigma) & \cosh(\Sigma) \end{array} \right)} \left( \begin{array}{cc} U^T & 0 \\ 0 & V^T  \end{array} \right)$$ $$=  \left( \begin{array}{cc}
\cosh{\left(\sqrt{AA^T}\right)} & A\left(\sqrt{A^TA }\right)^{\dagger}\sinh{\left(\sqrt{A^TA }\right)} \\
\sinh{\left(\sqrt{A^TA}\right)}\left(\sqrt{A^TA}\right)^{\dagger}A^T & \cosh{\left(\sqrt{A^TA}\right)}  \end{array} \right).$$
The identities involving the off-diagonal blocks can be easily checked using
the SVD of $A$.
\end{proof}

\subsection{Interpretation of diagonal entries}
\label{sec:interp}
In the context of undirected networks, 
the interpretation of the entries of the matrix exponential 
in terms of subgraph centralities and communicabilities is well-established,
see e.g.~\cite{NetworkProp}.  
In the case of directed networks and $e^{\mathcal{A}}$, things are not as clear.  The network behind $\mathcal{A}$ can be thought of as follows: take the vertices from the original network $A$ and make two copies of them, $V$ and $V'$.  Then, undirected edges exist between the two sets based on the following rule: $E' = \{(i,j') | {\rm \, there \, is \, a \, directed\, edge \, from \,} i {\rm \,to\,} j {\rm \,in \,the \,}  {\rm original \,network}\}$.  This creates a bipartite graph with $2n$ nodes: $1, 2, \ldots ,n,n+1,n+2, \ldots ,2n$. We denote by $V(\mathcal{A})$ this set of
nodes. The use of bipartization to treat rectangular and structurally
unsymmetric matrices is of course standard in numerical linear algebra.

In the undirected case, each node had only one role to play in the network: any information that came into the node could leave by any edge.  In the directed case, there are two roles for each node: that of a hub and that of an authority.  It is unlikely that a high ranking hub will also be a high ranking authority, but each node can still be seen as acting in both of these roles.  In the network $\mathcal{A}$, the two aspects of each node are separated.  Nodes 
$1, 2,\ldots ,n$ in $V(\mathcal{A})$ represent the original nodes 
in their role as hubs and nodes $n+1, n+2,\ldots, 2n$ in $V(\mathcal{A})$ represent the original nodes in their role as authorities.

Given a directed network, an {\em alternating walk} of length $k$, starting with an out-edge, from node $v_1$ to node $v_{k+1}$  is a list of nodes $v_1, v_2, ..., v_{k+1}$ such that there exists edge $(v_i, v_{i+1})$ if $i$ is odd and edge $(v_{i+1}, v_i)$ if $i$ is even:
$$v_1 \rightarrow v_2 \leftarrow v_3 \rightarrow \cdots $$
An {\em alternating walk} of length $k$, starting with an in-edge, from node $v_1$ to node $v_{k+1}$ in a directed network is a list of nodes $v_1, v_2, ..., v_{k+1}$ such that there exists edge $(v_{i+1}, v_i)$ if $i$ is odd and edge $(v_i, v_{i+1})$ if $i$ is even:
$$v_1 \leftarrow v_2 \rightarrow v_3 \leftarrow \cdots $$
From graph theory (see also \cite{croftetal}), it is known that $[AA^TA \ldots]_{ij}$ (where there are $k$ matrices being multiplied) counts the number of alternating walks of length $k$, starting with an out-edge, from node $i$ to node $j$, 
whereas  $[A^TAA^T\ldots]_{ij}$ (where there are $k$ matrices being multiplied) counts the number of alternating walks of length $k$, starting with an in-edge, from node $i$ to node $j$.  That is, $[(AA^T)^k]_{ij}$ and $[(A^TA)^k]_{ij}$ count the number of alternating walks of length $2k$.

In the original network $A$, if node $i$ is a good hub, it will point to many good authorities, which will in turn be pointed at by many hubs.  These hubs will also point to many authorities, which will again be pointed at by many other hubs.  Thus, if $i$ is a good hub, it will show up many times in the sets of hubs described above.  That is, there should be many even length alternating walks, starting with an out-edge, from node $i$ to itself.  Giving a walk of length $2k$ a weight of 
$\frac{1}{(2k)!}$, these walks can be counted using the $(i,i)$ entry of the matrix
$$I + \frac{AA^T}{2!} + \frac{AA^TAA^T}{4!} + \cdots + \frac{(AA^T)^{k}}{(2k)!}+ \cdots$$
Letting $A = U\Sigma V^T$ be the SVD of $A$, this becomes:
%$$U\left(I + \frac{AA^T}{2!} + \frac{AA^TAA^T}{4!} + \cdots + 
%\frac{(AA^T)^{k}}{(2k)!}+ \cdots\right)U^T$$ $$= U\cosh(\Sigma)U^T 
%= \cosh(\sqrt{AA^T})$$
$$U\left(I + \frac{\Sigma ^2}{2!} + \frac{\Sigma ^4}{4!} + \cdots +
\frac{\Sigma^{2k}}{(2k)!}+ \cdots\right)U^T$$ $$= U\cosh(\Sigma)U^T
= \cosh(\sqrt{AA^T})\,.$$

The {\em hub centrality} of node $i$ (in the original network) 
is thus given by $$[e^{\mathcal{A}}]_{ii} = [\cosh(\sqrt{AA^T})]_{ii}.$$  
This measures how well node $i$ transmits information to the 
authoritative nodes in the network.

Similarly, if node $i$ is a good authority, there will be many even length alternating walks, starting with an in-edge, from node $i$ to itself.  Giving a walk of length $2k$ a weight of $\frac{1}{(2k)!}$, these walks can be counted using the $(i,i)$ entry of $\cosh(\sqrt{A^TA})$.  

Hence, the {\em authority centrality} of node $i$ is given by 
$$[e^{\mathcal{A}}]_{n+i,n+i}= [\cosh(\sqrt{A^TA})]_{ii}.$$  
It measures how well node $i$ receives information from the hubs in the network.

Note that the traces of the two diagonal blocks in $e^{\cal A}$ are identical,
so each accounts for half of the Estrada index of the bipartite graph.
Also, recalling the well-known fact that the eigenvalues of $\cal A$
are $\pm \sigma_i$ where $\sigma_i$ denotes the singular values of $A$,
we have 
$${\rm Tr}\,(e^{\cal A}) = 
\sum_{i=1}^n e^{\sigma_i} + \sum_{i=1}^n e^{-\sigma_i} 
 = 2\sum_{i=1}^n \cosh\,(\sigma_i),$$
an identity that can also be obtained directly from the
expression for $e^{\cal A}$ given in Proposition \ref{prop1}.

Returning to the example of the directed path graph with adjacency
matrix $A$ given by (\ref{AP}), one finds that using the diagonal entries of
$e^{\cal A}$ to rank the nodes gives node 1 as the least authoritative
node, and node $n$ as the one with the lowest hub ranking, with all
the other nodes being tied. Thus we see that, while $e^A$ fails to 
differentiate between the nodes of this graph, using $e^{\cal A}$
yields a very reasonable hub/authority ranking of the nodes.

\subsection{Interpretation of off-diagonal entries}
Although not used in the remainder of this paper, for the sake of
completeness we give here
an interpretation of the off-diagonal entries
of $e^{\cal A}$. As we will see, this interpretation is
rather different from the one usually given for the off-diagonal
entries of $e^A$, and provides information of a different nature
on the structure of the underlying graph. 

In discussing the off-diagonal entries of $\mathcal{A}$, there are three 
blocks to consider.  First, there are the off-diagonal entries of the upper-left block, $\cosh(\sqrt{AA^T})$, then there are the off-diagonal entries of the lower-right block, $\cosh(\sqrt{A^TA})$.  Finally, there is the off-diagonal block, 
$ A\left(\sqrt{A^TA }\right)^{\dagger}\sinh \left(\sqrt{A^TA }\right)$
(the fourth block in $e^{\cal A}$ being its transpose).

From section \ref{sec:interp}, 
$[e^{\mathcal{A}}]_{ij} = [\cosh(\sqrt{AA^T})]_{ij}$, $1 \leq i,j \leq n$, counts the number of even length alternating walks, starting with an out-edge, from node $i$ to node $j$, weighting walks of length $2k$ by a factor of $\frac{1}{(2k)!}$.  When $i \neq j$, these entries measure how similar nodes $i$ and $j$ are as hubs.  That is, if nodes $i$ and $j$ point to many of the same nodes, there will be many short even length alternating walks between them.

The {\em hub communicability} between nodes $i$ and $j$, $1 \leq i,j \leq n$, is given by 
$$[e^{\mathcal{A}}]_{ij} = [\cosh(\sqrt{AA^T})]_{ij}$$
This measures how similar nodes $i$ and $j$ are in their roles as hubs.  That is, a larger value of hub communicability between nodes $i$ and $j$ indicates that they point to many of the same authorities.  In other words, they point to nodes which are authorities on the same subjects.

Similarly, $[e^{\mathcal{A}}]_{n+i,n+j} = [\cosh(\sqrt{A^TA})]_{ij}$, $1 \leq i,j \leq n$, counts the number of even length alternating walks, starting with an in-edge, from node $i$ to node $j$, also weighing walks of length $2k$ by a factor of $\frac{1}{(2k)!}$.  When $i \neq j$, these entries measure how similar the two nodes are as authorities.  If $i$ and $j$ are pointed at by many of the same hubs, there will be many short even length alternating walks between them.

The {\em authority communicability} between nodes $i$ and $j$, $1 \leq i,j, \leq n$, is given by 
$$[e^{\mathcal{A}}]_{i+n,j+n} = [\cosh(\sqrt{A^TA})]_{ij}$$
This measures how similar nodes $i$ and $j$ are in their roles as authorities.  That is, a larger value of authority communicability between nodes $i$ and $j$ means that they are pointed to by many of the same hubs and, as such, are likely to contain information on the same subjects.

Let us now consider the off-diagonal blocks of $\mathcal{A}$. 
Here, $[\sinh(\sqrt{A^TA})]_{ij}$ counts the number of odd length alternating walks, starting with an out-edge, from node $i$ to node $j$, weighing walks of length 
$2k+1$ by $\frac{1}{(2k+1)!}$.  This measures the communicability between node $i$ as a hub and node $j$ as an authority.

The {\em hub-authority communicability} between nodes $i$ and $j$ (that is, the communicability between node $i$ as a hub and node $j$ as an authority) is given by:
$$[e^{\mathcal{A}}]_{i,n+j} = [A\left(\sqrt{A^TA }\right)^{\dagger}\sinh \left(\sqrt{A^TA }\right)]_{ij} $$ $$= [\sinh{\left(\sqrt{A^TA}\right)}\left(\sqrt{A^TA}\right)^{\dagger}A^T]_{ji} = [e^{\mathcal{A}}]_{n+j, i}.$$
A large hub-authority communicability between nodes $i$ and $j$ means that they are likely in the same \lq\lq part\rq\rq\ of the directed network: node $i$ tends to point to nodes that contain information similar to that on which node $j$ is an authority.

\subsection{Relationship with HITS}\label{hits_comp}
As described in \ref{sec:newHITS}, the HITS ranking of nodes 
as hubs and authorities uses only information from the dominant 
eigenvector of $\mathcal{A}$. Here we show that when using the
diagonal of $e^{\mathcal A}$, we exploit information contained
in all the eigenvectors of $\mathcal{A}$; moreover, the HITS rankings
can be regarded as an approximation of those given by the
diagonal entries of $e^{\mathcal A}$.

Assume the eigenvalues of $\mathcal{A}$ can be ordered as
$\lambda_{1} > \lambda_{2} \geq \lambda_{3} \geq \cdots \geq 
\lambda_{2n}$.
Then, $\mathcal{A}$ can be written as $\mathcal{A} = \sum_{i=1}^{2n} 
\lambda_iu_iu_i^T$ where $u_1, u_2,\ldots, u_{2n}$ are the 
normalized eigenvectors of $\mathcal{A}$.  Taking the exponential of 
$\mathcal{A}$, we get: 
$$e^{\mathcal{A}} = \sum_{i=1}^{2n} 
e^{\lambda_i}u_iu_i^T = e^{\lambda_{1}}u_1u_1^T + \sum_{i=2}^{2n} 
e^{\lambda_i}u_iu_i^T.$$

Now, the hub and authority rankings come from the diagonal entries 
of $e^{\mathcal{A}}$: 
$${\rm diag}\,(e^{\mathcal{A}}) = e^{\lambda_{1}}
{\rm diag}\,(u_1u_1^T) + \sum_{i=2}^{2n}e^{\lambda_i}{\rm diag}\,(u_iu_i^T).$$ 
Rescaling the hub and authority scores by $e^{\lambda_{1}}$ does not alter
the rankings; hence, we can instead consider
$${\rm diag}\,(e^{-\lambda_{1}}e^{\mathcal{A}}) = 
{\rm diag}\,(e^{\mathcal{A}-\lambda_{1}I}) =
{\rm diag}\,(u_1u_1^T) + \sum_{i=2}^{2n}e^{\lambda_i-\lambda_{1}}{\rm diag}\,(u_iu_i^T).$$
Now, the diagonal entries of the rank-one matrix $u_1u_1^T$ are 
just the squares of the (nonnegative) entries of the dominant eigenvector of
$\mathcal{A}$; hence,
the rankings provided by the first
term in the expansion of $e^{\mathcal{A}}$ in the eigenbasis of $\cal A$ are 
precisely those given by HITS.

It is also clear that if $\lambda_{1} \gg \lambda_{2}$, 
then the rankings provided
by the diagonal entries of $e^{\mathcal{A}}$ are unlikely to differ much from
those of HITS, since the weights $e^{\lambda_i-\lambda_{1}}$ will be tiny, for
all $i=2,\ldots ,2n$. 
 Conversely, if the gap between $\lambda_{1}$ and the rest of the
spectrum is small ($\lambda_{1}\approx \lambda_{2}$), then the contribution 
from the remaining eigenvectors,
$\sum_{i=2}^{2n}e^{\lambda_i-\lambda_{1}}{\rm diag}\,(u_iu_i^T)$, may be 
non-negligible relative to the first term and therefore the resulting rankings
could differ significantly from those obtained using HITS. %This is true, 
%{\em a fortiori}, if the dominant eigenvalue of $\cal A$ is not simple. 
In section \ref{sec:applications} we will see examples of real networks illustrating
both scenarios.

Summarizing, use of the matrix exponential for ranking hubs and authorities
amounts to using the (squared) entries of {\em all} the eigenvectors of $\cal A$,
weighted by the exponential of the corresponding eigenvalues. 
Of course, in place of the exponential, a number of other functions
could be used; see the discussion in the next section. 
Although using an exponential
weighting scheme may at first sight appear to be arbitrary, its use can be
rigorously justified; see the discussion in the next section,
and \cite{EHB11} for a thorough treatment in
the context of undirected graphs. As shown above, the HITS ranking scheme 
uses the leading term only, corresponding to the approximation 
$e^{\mathcal A}\approx e^{\lambda_1}u_1u_1^T$. Between these two extremes
one could also use approximations of the form
\begin{equation}\label{k_terms}
e^{\mathcal A} \approx \sum_{i=1}^k e^{\lambda_i}u_iu_i^T\,,
\end{equation}
where $1< k < n$; indeed, in most cases of practical interest a modest
value of $k$ ($\ll n$)
 will be sufficient for a very good approximation, since the
eigenvalues of $\cal A$ are often observed to decay rapidly from a certain
index $k$ onward. We return on this topic in section \ref{sec:approx}.

\section{Other ranking schemes}
In this section we discuss a few other schemes that have been proposed
in the literature, and compare them with the hub and authority centrality
measures based on the exponential of $\cal A$.

\subsection{Resolvent-based measures}
Besides the matrix exponential, another function that has been successfully
used to define centrality and communicability measures for an undirected
network is the matrix
resolvent, which can be defined as
$$R(A;c) = (I - c A)^{-1} = I + cA + c^2 A^2 +\cdots + c^kA^k + \cdots  \,,$$
with $0 < c< 1/\lambda_{\max}(A)$. This approach was pioneered early
on by Katz \cite{Katz}, and variants thereof have since been used 
by numerous authors; see, e.g., 
\cite{bonchietal,Brandes,EHB11,NetworkProp,franceschet,vigna}.
Here $A$ is the symmetric adjacency matrix of the undirected network.
The condition on the parameter $c$ ensures that $R(A;c)$ is well 
defined (i.e., that $I - c A$ is invertible and the geometric series
converges to its inverse) and nonnegative; 
indeed, $I - cA$ will be a nonsingular $M$-matrix. It is
hardly necessary to mention the close relationship existing between the
resolvent and the exponential function, which can be expressed via the
Laplace transform. For the adjacency matrix $\cal A$ of a bipartite
graph given by (\ref{calA}), the resolvent
is easily determined to be
\begin{equation}\label{resolv}
R({\cal A};c) = 
\left( \begin{array}{cc} (I - c^2AA^T)^{-1} & cA(I - c^2A^TA)^{-1}\\
                         c(I - c^2A^TA)^{-1}A^T     & (I - c^2A^TA)^{-1}
                         \end{array} \right).
\end{equation}
The condition on $c$ can be expressed as $0 < c < 1/\sigma_1$, 
where $\sigma_1=\|A\|_2$
denotes the largest singular value of $A$, the adjacency matrix of the undirected 
network. This ensures that the matrix in (\ref{resolv}) is well-defined and
 nonnegative, with
positive diagonal entries. The diagonal entries of $(I - c^2AA^T)^{-1}$
provide the hub scores, those of $(I - c^2A^TA)^{-1}$ the authority scores. 
A drawback of this approach is the need to select
the parameter $c$, and the fact that different values of $c$ may lead
to different rankings. We have performed numerical experiments with this
approach and we found that for certain values of $c$, particularly 
those close to the upper limit $1/\sigma_1$, the hub and authority
rankings obtained with the resolvent function are not too different
from those obtained with the matrix exponential. However,
not surprisingly, as the value of $c$ is reduced, one obtains hub and
authority rankings that
are strongly correlated with the out- and in-degree of the nodes,
respectively.\footnote{Note that if $c$ is taken too small, then the resolvent
approaches the identity matrix and it becomes impossible to have 
meaningful rankings of the nodes.}
 Overall, because the resolvent tends to weigh short walks more
heavily than the exponential, and since longer walks contribute relatively
little to the centrality scores, it is fair to say that the exponential
is less \lq\lq degree biased\rq\rq\ than the resolvent function. Also,
since the exponential rankings do not depend on a tuneable parameter,
they provide unambiguous rankings.

%Other functions that have been used for the analysis of complex networks include 
%variants of the exponential, such as $f(A) = e^{\beta A}$, where $\beta$ can be
%interpreted as a (negative) inverse \lq\lq temperature\rq\rq, 
%as well as similar functions
%involving the graph Laplacian $L= D-A$, where $D$ is the degree matrix,
%$D={\rm diag}\,(d_1,\ldots ,d_n)$. We refer to \cite{EHB11} for a detailed study 
%and justification of these matrix functions in the study of undirected networks.
%A comparison of the various hub and authority rankings obtained using
%these functions is beyond the scope of this paper, and will be
%the subject of a separate study.

We note that
\lq\lq Katz\rq\rq\  authority and hub scores may also be obtained by considering 
the column and row sums of the (nonsymmetric) matrix resolvent
$(I - c A)^{-1}$, where $A$ is the adjacency
matrix of the original digraph and $c>0$ is again assumed to
be small enough for the corresponding Neumann series to converge. 
Indeed, the row sums of $(I - c A)^{-1}$ count the number of
(weighted) walks out of each node, while the column sums count the
number of (weighted) walks into each node. Denoting by $\bf 1$ the
vector of all ones, hub and authority rankings can be obtained by solving the
two linear systems
\begin{equation}\label{katz}
(I - c A)y = {\bf 1} \quad {\rm and} \quad (I - c A^T)x = {\bf 1}\,,
\end{equation}
respectively. Here the parameter $c$ must satisfy $0 < c < 1/\rho(A)$,
where $\rho(A)$ denotes the spectral radius of $A$.
The results of numerical experiments comparing the Katz scores with those
based on the exponential of $\cal A$ are given in section \ref{sec:applications}.
Here we observe that these Katz scores are also dependent on the choice
of the parameter $c$, and similar considerations to those made for
$(I - c {\cal A})^{-1}$ apply.

A natural analogue to this approach is the use 
of row and column sums of the
exponential $e^A$ to rank hubs and authorities. Some results obtained with 
this approach are discussed in section \ref{sec:applications}. 
We note that this method 
is different from the {\em Exponentiated Inputs 
HITS Method} of \cite{EIHITS}.  
The latter method is a modification to HITS which was developed 
in order to correct the issue of non unique results in 
certain networks.  If the dominant eigenvalue of $A^TA$ 
(and, consequently, of $AA^T$) is not simple, then the 
corresponding eigenspace is multidimensional. This means 
that the choice of the initial vector affects the convergence 
of the HITS algorithm and different hub and authority vectors 
can be produced using different initial vectors. This can 
occur only when $A^{T} A$ is reducible, that is, when the original 
network is not strongly connected. In \cite{EIHITS}, 
Farahat {\em et al.}~propose a modification to the HITS 
algorithm which amounts to replacing $A$ and $A^T$ with 
$e^A-I$ and $(e^A-I)^T$ in the HITS iteration.  They note 
that, as long as the original network is weakly connected, 
the dominant eigenvalue of $(e^A-I)^T (e^A - I)$ is simple. 
Thus, HITS with this exponentiated input produces unique 
hub and authority rankings.  However, a result of this 
replacement is that nodes with zero in-degree (or a low in-degree) 
are less important in the calculation of authority scores 
than nodes with a high in-degree.   When there are many nodes 
with zero in-degree or whose edges point to only a few other nodes, 
dropping these edges can greatly affect the HITS rankings.
An obvious disadvantage of this algorithm is its cost, 
since it requires iterating with a matrix exponential and 
its transpose. It can be implemented using only matrix-vector 
products involving $A$ and $A^T$ by means 
of techniques, like Krylov subspace methods, for evaluating 
the action of a matrix function on a given vector; see, 
e.g., \cite[Chapter 13]{highambook}. This approach leads 
to a nested iteration scheme, with HITS as the outer iteration 
and the Krylov method as the inner one. 
Generally speaking, we have found HITS with exponentiated inputs
to be less reliable and more expensive than the other methods
considered in this paper. We refer
to \cite{TechReport} for additional discussion and
some examples.

\subsection{PageRank and Reverse PageRank}
As is well known, the (now) classical PageRank algorithm 
provides a means of finding
the authoritative nodes in a
digraph. In PageRank, the importance of a node $v$ is determined 
by the importance of the nodes pointing to $v$.  In the most 
basic formulation, the rank of $v$ is given by 
\begin{equation}
	r(v) = \sum_{u \in B_v} \frac{r(u)}{|u|}
\end{equation}
where $B_v = \{u:$ there is a directed link from $u$ to $v \}$ 
and $|u|$ is the out-degree of $u$.  The ranks of the nodes 
are computed by initially setting, say, $r(v) = \frac{1}{n}$ 
(where $n$ is the size of the network) and iteratively 
computing the rankings until convergence.  This can also be written as
\begin{equation}
	\pi_k^T = \pi_{k-1}^T P, \quad k=1,2,\ldots
\end{equation}
where $\pi_k$ is the vector of node ranks at iteration 
$k$ and $P$ is the matrix given by $$p_{ij}=\left\{\begin{array}{ll}
1/|v_i|,& \textnormal{ if  there is a directed edge from } 
v_i \textnormal{ to } v_j,\\
0, & \textnormal{ else. }
\end{array}\right .
$$
Here, $P$ can be viewed as a probability transition matrix, 
where $p_{ij}$ is the probability of traveling from node $v_i$ 
to node $v_j$ along an edge and the iterations can be understood 
as the evolution of a Markov chain modeling a random walk on 
the network.

However, for an arbitrary network, there is no guarantee 
that the PageRank algorithm will converge.  If there are 
nodes with zero out-degree, $P$ will not be stochastic.  
To correct this, the matrix $\bar{P}$ is used, where each 
zero row of $P$ is replaced with ${\bf e}^T/n$.  Although 
this guarantees that the algorithm will converge, it does 
not guarantee the existence of a unique solution.  Even 
with the augmentation, $\bar{P}$ might still be a reducible 
matrix, corresponding to a reducible Markov chain.  When this 
happens, there are rank sinks, i.e., nodes in which the 
random walk will become trapped and, subsequently, these 
nodes will receive a disproportionately high rank.  However, 
if $P$ were irreducible, there would be no rank sinks and 
the Perron-Frobenius theorem would guarantee that the Markov 
chain had a unique, positive stationary distribution.

The standard way to form a stochastic, irreducible 
PageRank matrix $\bar{\bar{P}}$ is to introduce the
rank-1 matrix $E={\bf ee}^T/n$ and to consider
instead of $\bar P$ the convex combination
\begin{equation}
	\bar{\bar{P}} = \alpha\bar{P} +(1-\alpha)E\,,
\end{equation}
where $\alpha$ is a constant with $0 < \alpha <1$.  
The coefficient $1-\alpha$ is a measure of the tendency 
of a person surfing the web to jump from one page to 
another without following links.  In practice, a 
frequently recommended value is $\alpha =0.85$.  
For a more comprehensive overview of the PageRank 
algorithm, see \cite{franceschet,Kam,IRSurvey,Pagerank}. 
 
It was pointed out in \cite{fogaras} that applying PageRank to the 
digraph obtained by reversing the direction of the edges provides a
natural way to rank the hubs; this is usually referred to as
{\em Reverse PageRank}. In other words, authority rankings are
obtained by applying PageRank to the \lq\lq Google\rq\rq\ matrix 
derived from $A$, and hub rankings are obtained by the same process
applied to $A^T$. 
 Like HITS, PageRank and Reverse PageRank are
eigenvector-based ranking algorithms that do not take into account
information about the network contained in the non-dominant eigenvectors.
As already mentioned,
it has been argued \cite{MiPapa} that eigenvector-based algorithms
tend to be degree-biased. 
Furthermore, like the Katz-type algorithms, the rankings obtained
depend on the choice of a tuneable damping parameter.
While the success of PageRank in finding authoritative nodes is
well known and very well documented, much less is known about the
effectiveness of Reverse PageRank in identifying hubs; some references
are \cite{Bar-Yossef,crofootetal,WuLi,WuLietal}. We present
the results of a few numerical experiments with PageRank and Reverse PageRank
in section \ref{sec:applications}.

\section{Examples}\label{expes}
In this section and the next we illustrate the proposed method on some simple networks
of small size, as well as on some larger data sets corresponding
to real networks. We also compare our approach with HITS and other
rankings schemes, including Katz, PageRank and Reverse PageRank.

\subsection{Small digraphs}
In this section we compare out and in-degree counts, HITS, and
our proposed method to obtain hub and authority rankings in a 
few small digraphs. The purpose of this section is mostly
pedagogical.

\subsubsection{Example 1}
Consider the small directed network in Fig.~\ref{graphs} (left panel).
The adjacency matrix is given by
$$A =  \left( \begin{array}{cccc} 0 & 1 & 1 & 0 \\ 1 & 0 & 1 &0 \\ 0 & 1 & 0 & 1 \\ 0 & 1 & 0 & 0 \end{array} \right).$$

\begin{figure}[t!]
\centering
\includegraphics[width=0.30\textwidth]{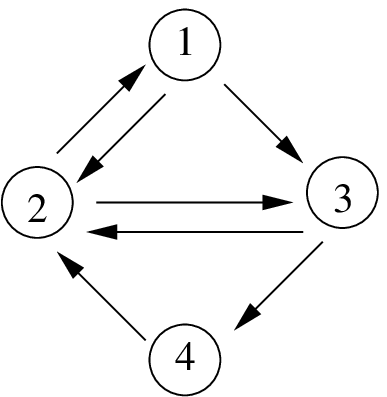}
\hspace{0.3in}
\includegraphics[width=0.22\textwidth]{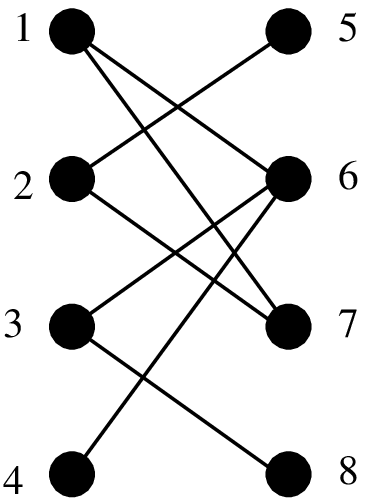}
\caption{The original directed network from Example 1, with adjacency matrix $A$ (left) 
and the bipartite network with adjacency matrix $\mathcal{A}$ (right).}
\label{graphs}
\end{figure}

%\begin{figure}[ht]
%\centering
%\subfigure[The original directed network, with adjacency matrix $A$.]{
%%\includegraphics[width=0.25\textwidth]{digraph.eps}
%\includegraphics[width=0.25\textwidth]{digraph}
%\label{fig:subfig1}
%}
%%\hspace{.3in}
%%\subfigure[The bipartite network with adjacency matrix $\mathcal{A}$]{
%\includegraphics[width=0.18\textwidth]{bipartite.eps}
%\includegraphics[width=0.18\textwidth]{bipartite}
%\label{fig:subfig2}
%}
%\label{fig:subfigureExample}
%\caption[Optional caption for list of figures]{Networks behind $A$ and $\mathcal{A}$.}
%\end{figure}

The corresponding bipartite graph is shown in  Fig.~\ref{graphs} (right panel).
If hubs and authorities are determined simply using in-degree and out-degree 
counts, the result is as follows: 

\vspace{0.12in}

\begin{center}
\begin{tabular}{|c|c|c|}
\hline
node & out-degree & in-degree \\ \hline
1 & 2 & 1\\
2 & 2 & 3\\
3 & 2 & 2\\
4 & 1 & 1\\
\hline
\end{tabular}
\end{center}

\vspace{0.12in}

Under this ranking, the hub ranking of the nodes is: 
$\{1, 2, 3 {\rm \,(tie)}; 4\}$.  The authority ranking 
of the nodes is: $\{2;3; 1,4 {\rm \, (tie)}\}$. %These rankings
%are not satisfactory as they result in many nodes being tied.
We obtain somewhat different results
using the HITS algorithm. The 
eigenvectors of $AA^T$ and $A^TA$ corresponding to the 
largest eigenvalue $\lambda_{\max} \approx 3.9563$, which is
simple, yield the
following rankings for hubs and authorities:

\vspace{0.12in}

\begin{center}
\begin{tabular}{|c|c|c|}
\hline
node & hub rank & authority rank \\ \hline
1 & .3383 & .0965 \\
2 & .1729 & .4618\\
3 & .2798 & .2854\\
4 & .2091 & .1562\\
\hline
\end{tabular}
\end{center}

\vspace{0.12in}

Here, the ranking for hubs is: $\{1;3;4;2\}$. 
The ranking for authorities is: $\{2;3;4;1\}$.
Note that node 2, which was given a top
hub score by looking just at the out-degrees, is judged
by HITS as the node with the lowest hub score. 

Using $e^{\mathcal{A}}$ as described above, the 
rankings for hub centralities and authority centralities are:

\vspace{0.12in}

\begin{center}
\begin{tabular}{|c|c|c|}
\hline
node & hub centrality = $[e^{\mathcal{A}}]_{ii}$ & authority centrality 
= $[e^{\mathcal{A}}]_{4+i,4+i}$  \\ \hline
1 & 2.3319 & 1.5906\\
2 & 2.2289 & 3.0209\\
3 & 2.2812 & 2.2796\\
4 & 1.6414 & 1.5922\\
\hline
\end{tabular}
\end{center}

\vspace{0.12in}

With this method, the hub ranking of the nodes is: $\{1;3;2;4\}$.  
The authority ranking is: $\{2;3;4;1\}$. On this example, our method produces  
the same authority ranking as HITS. The hub ranking, however, is slightly
different: both methods identify node 1 as the one with the highest hub
score, followed by node 3; however, our method assigns the lowest hub score
to node 4 rather than node 2. This is arguably a more meaningful
ranking. 

\subsubsection{Example 2}
Consider the small directed network in Fig.~\ref{graphs2} (left panel). The adjacency matrix is given by 
$$A =  \left( \begin{array}{cccc} 0 & 0 & 1 & 0 \\ 1 & 0 & 0 &1 \\ 0 & 1 & 0 & 0 \\ 0 & 1 & 0 & 0 \end{array} \right).$$

\begin{figure}[t!]
\centering
\includegraphics[width=0.30\textwidth]{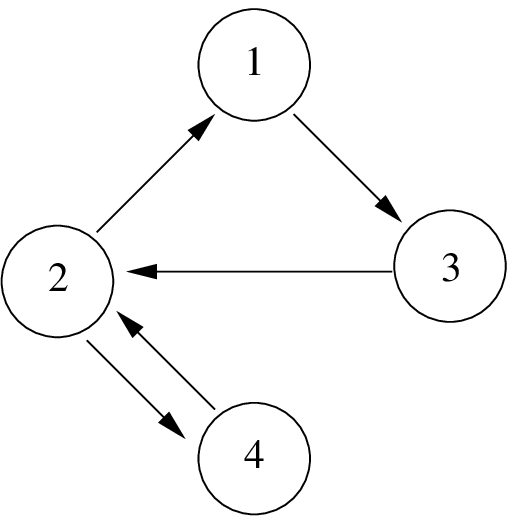}
\hspace{0.3in}
\includegraphics[width=0.21\textwidth]{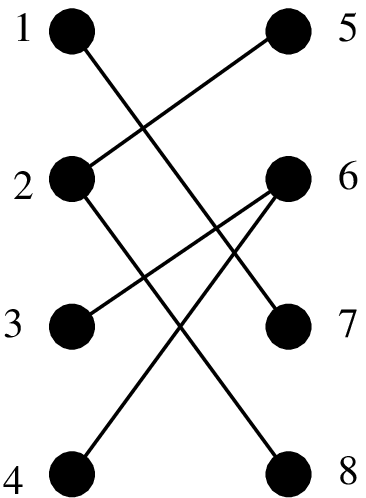}
\caption{The original directed network from Example 2, with adjacency matrix $A$ (left) 
and the bipartite network with adjacency matrix $\mathcal{A}$ (right).}
\label{graphs2}
\end{figure}

The corresponding bipartite graph is shown in Fig.~\ref{graphs2} (right panel).  If hubs and authorities are determined only using in-degrees and out-degrees, the result is as follows:

\vspace{0.12in}

\begin{center}
\begin{tabular}{|c|c|c|}
\hline
node & out-degree & in-degree \\ \hline
1 & 1 & 1\\
2 & 2 & 2\\
3 & 1 & 1\\
4 & 1 & 1\\
\hline
\end{tabular}
\end{center}

\vspace{0.12in}

Under this criterion, the hub and authority rankings are
both %given by
$\{2; 1, 3, 4\, {\rm (tie)}\}$.
While it is intuitive that 
node 2 should be given a high score (both as an authority
and as a hub), just looking at the degrees does not allow
one to distinguish the remaining nodes.

Consider now the use of HITS.  The largest eigenvalue 
of $AA^T$ (and $A^TA$) is $\lambda_{\max} =2$ and it has multiplicity two.  
Thus, different starting vectors for the HITS algorithm may produce different rankings,
as discussed in \cite{EIHITS}.  
Starting from a constant authority vector $x^{(0)}$,
as suggested in \cite{HITS},
produces the following scores:

\vspace{0.12in}

\begin{center}
\begin{tabular}{|c|c|c|}
\hline
node & hub rank & authority rank \\ \hline
1 & .0000 & .3333\\
2 & .5000 & .3333\\
3 & .2500 & .0000\\
4 & .2500 & .3333\\
\hline
\end{tabular}
\end{center}

\vspace{0.12in}

\noindent The ranking for hubs is: $\{2; 3, 4\, {\rm (tie)}; 1\}$.  The ranking for 
authorities is the following: 
$\{1, 2, 4\, {\rm (tie)}; 3\}$. 
%The fact that node 2 has a higher in-degree than any 
%of the other nodes does not 
%affect its authority ranking under the HITS algorithm, which shows a clear
%drawback of HITS.

If the ranking is determined using $e^{\mathcal{A}}$ as described above, 
the resulting scores are:

\vspace{0.12in}

\begin{center}
\begin{tabular}{|c|c|c|}
\hline
node & hub centrality = $[e^{\mathcal{A}}]_{ii}$ & authority 
centrality = $[e^{\mathcal{A}}]_{4+i,4+i}$  \\ \hline
1 & 1.5431 & 1.5891\\
2 & 2.1782 & 2.1782\\
3 & 1.5891 & 1.5431\\
4 & 1.5891 & 1.5891\\
\hline
\end{tabular}
\end{center}

\vspace{0.12in}

\noindent With this method, the hub ranking of the nodes is the same as in HITS: 
$\{2; 3,4\, {\rm (tie)}; 1\}$.  However, in the authority ranking, 
node 2 is the clear winner rather than being part of a three-way 
tie for first place: $\{2; 1, 4\, {\rm (tie)}; 3\}$. %Identical rankings
%are obtained by the exponentiated input HITS algorithm. 
In this example, the method based on the matrix exponential is
able to identify a top authority node by making
use of additional spectral information. 

\subsection{Example 3}
\label{example3}
Let $G$ be the small directed network in Fig.~\ref{graphs3}.  The adjacency matrix is given by
$$A =  \left( \begin{array}{cccccc} 0 & 0 & 0 & 0 & 0 & 0 \\ 1 & 0 & 0 & 0 & 0 & 0 \\ 1 & 0 & 0 & 0 & 0 & 0 \\ 1 & 0 & 0 & 0 & 0 & 0 \\ 1 & 0 & 0 & 0 & 0 & 0 \\ 0 & 1 & 1 & 1 & 1 & 0 \end{array} \right).$$

\begin{figure}[t!]
\centering
\includegraphics[width=0.35\textwidth]{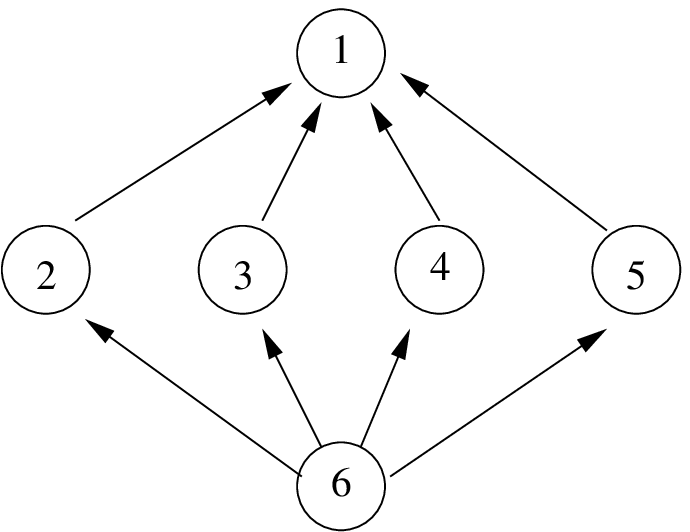}
\hspace{0.3in}
\includegraphics[width=0.2\textwidth]{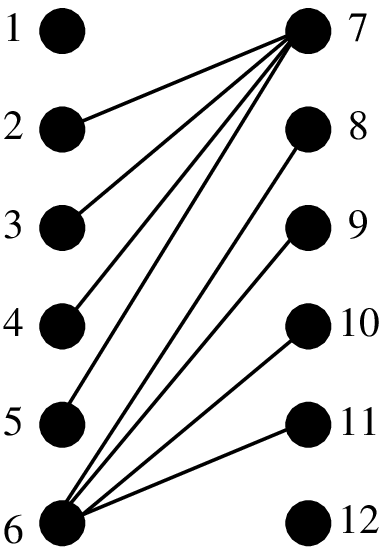}
\caption{The original directed network from Example 3, with adjacency matrix $A$ (left)
and the bipartite network with adjacency matrix $\mathcal{A}$ (right).}
\label{graphs3}
\end{figure}

If hubs and authorities are determined using only in-degrees and out-degrees, the result is:

\vspace{0.12in}

\begin{center}
\begin{tabular}{|c|c|c|}
\hline
node & out-degree & in-degree \\ \hline
1 & 0 & 4\\
2 & 1 & 1\\
3 & 1 & 1\\
4 & 1 & 1\\
5 & 1 & 1\\
6 & 4 & 0\\
\hline
\end{tabular}
\end{center}

\vspace{0.12in}

The hub ranking of the nodes using degrees is: \{6; 2,3,4,5 (tie); 1\}.  The authority ranking is \{1; 2,3,4,5 (tie); 6\}.

If the HITS algorithm is used, the resulting rankings are similar, 
but not exactly the same.  Starting with a constant authority vector 
%$x^{(0)} = [\frac{1}{6}, \frac{1}{6}, \frac{1}{6}, \frac{1}{6}, 
%\frac{1}{6}, \frac{1}{6}]$, the rankings produced are:
$x^{(0)}$, the results are:

\vspace{0.12in}

\begin{center}
\begin{tabular}{|c|c|c|}
\hline
node & hub rank & authority rank \\ \hline
1 & .000 & .200 \\
2 & .125 & .200 \\
3 & .125 & .200 \\
4 & .125 & .200 \\
5 & .125 & .200 \\
6 & .500 & .000 \\
\hline
\end{tabular}
\end{center}

\vspace{0.12in}

The hub ranking of the nodes is: $\{6; 2,3,4,5\, {\rm (tie)}; 1\}$.  
The authority ranking is: \{1,2,3,4,5 (tie); 6\}.  Here, HITS does
not differentiate
between node 1 and nodes 2, 3, 4, and 5 in 
terms of the authority score, even though node 1 has by far the highest 
in-degree.
This appears as a failure of HITS, since it is intuitive that node 1
should be regarded as very authoritative.
%If HITS with exponentiated input is used, node 1 does get a 
%higher authority score than all the other nodes:

%\vspace{0.12in}

%\begin{center}
%\begin{tabular}{|c|c|c|}
%\hline
%node & Exp.~input HITS hub rank & Exp.~input HITS authority rank \\ \hline
%1 & .0000 & .4472 \\
%2 & .1382 & .1382 \\
%3 & .1382 & .1382 \\
%4 & .1382 & .1382 \\
%5 & .1382 & .1382 \\
%6 & .4472 & .0000 \\
%\hline
%\end{tabular}
%\end{center}

%\vspace{0.12in}

%Under HITS with exponentiated input, the hub ranking is the
%same one obtained by HITS, namely, \{6; 2,3,4,5 (tie); 1\},  
%while the authority ranking is: \{1; 2,3,4,5 (tie); 6\}.

When $e^{\mathcal{A}}$ is used to calculate the hub and authority scores, node 1 does get a higher authority ranking than all the other nodes:

\vspace{0.12in}

\begin{center}
\begin{tabular}{|c|c|c|}
\hline
node & hub centrality = $[e^{\mathcal{A}}]_{ii}$ & authority 
centrality = $[e^{\mathcal{A}}]_{6+i,6+i}$  \\ \hline
1 & 1.0000 & 3.7622 \\
2 & 1.6905 & 1.6905 \\
3 & 1.6905 & 1.6905 \\
4 & 1.6905 & 1.6905 \\
5 & 1.6905 & 1.6905 \\
6 & 3.7622 & 1.0000 \\
\hline
\end{tabular}
\end{center}

\vspace{0.12in}

Note that, if desired, the value 1 can be subtracted from these 
scores since it does not affect the relative ranking of the nodes.
The hub ranking is \{6; 2,3,4,5 (tie); 1\}, and 
the authority ranking is: \{1; 2,3,4,5 (tie); 6\}.

\section{Application to web graphs}\label{sec:applications}
Similarly to HITS, and in analogy to subgraph centrality for
undirected networks,  the rankings produced by the values 
on the diagonal of $e^{\mathcal{A}}$ can be used to 
rank websites as hubs and authorities in web searches 
(many other applications are of course also possible).
 Three of the data sets considered here are small web graphs consisting of
web sites on various topics and can be found at \cite{datasets} 
along with the website associated with each node;
see also \cite{LinkStructure}. 
The experiments for this paper were run on the ``Expanded" version of the data sets. 
Each data set is named after the corresponding
topic.\footnote{It should be noted, 
however, that in the node list for the adjacency matrix, 
the node labeling begins with 1 and in the list of websites 
associated with the nodes found at \cite{datasets}, node 
labeling begins at 0.  Thus, node $i$ in the adjacency 
matrix is associated with website $i-1$.} 
In addition, we include results for the {\em wb-cs-stanford}
data set from the University of Florida Sparse Matrix Collection
\cite{UFL}. This digraph represents a subset of the
 Stanford University web. In this section, 
the hub and authority rankings obtained from 
$e^{\mathcal{A}}$ are compared with those from HITS, 
%implemented using the Matlab code described in \cite{Pagerank}.  
Katz (using (\ref{katz}) with $c=1/(\rho(A) + 0.1)$),
the row and column sums of the exponential $e^A$ of
the nonsymmetric matrix $A$, and PageRank/Reverse PageRank. 
For the latter we use the standard value $\alpha = 0.85$ for the damping
parameter.
All experiments are performed using Matlab Version 7.9.0 
(R2009b) on a MacBook Pro running 
OS X Version 10.6.8, a 2.4 GHZ Intel Core i5 processor 
and 4 GB of RAM. For the purpose of these tests we use the
built-in Matlab function {\tt expm} to compute the matrix
exponentials, and backslash to compute the
Katz scores. Other approximations of $e^{\mathcal{A}}$ are 
discussed in section \ref{sec:approx}.

\begin{figure}[t!]
\centering
\includegraphics[width=0.8\textwidth]{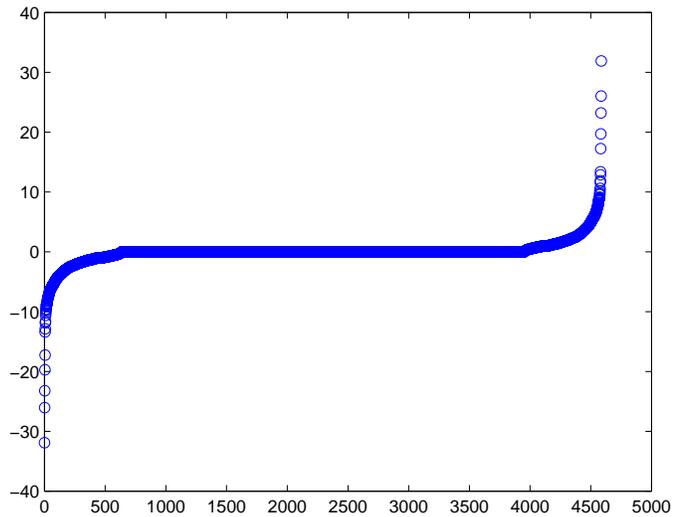}
\vspace{-.25in}
\caption{Plot of the eigenvalues of the expanded abortion matrix $\mathcal{A}$. }
\label{abortionevalues}
\end{figure}

\subsection{Abortion data set}

The {\em abortion} data set contains $n=2293$ nodes and $m=9644$ 
directed edges.  The expanded matrix $\mathcal{A} =  
\left( \begin{array}{cc} 0 & A \\ A^T & 0  \end{array} \right)$ 
has order $N=2n=4586$ and contains $2m=19288$ nonzeros.  
The maximum eigenvalue of $\mathcal{A}$ is $\lambda_{N} \approx 
31.91$ and the second largest eigenvalue is 
$\lambda_{N-1} \approx 26.04$.  
In this matrix, the largest eigenvalue is fairly 
well-separated from the second largest so that one would 
expect the HITS rankings (which only use information 
from the dominant eigenpair of $\cal A$) to be reasonably close to the 
rankings from $e^{\mathcal{A}}$ (which use information from 
all of the eigenvalues and corresponding eigenvectors). 
A plot of the eigenvalues of the expanded abortion data set 
matrix can be found in Fig.~\ref{abortionevalues}. 
Note the high multiplicity of the zero eigenvalue in this matrix, 
as well as in the adjacency matrices of the computational 
complexity and death penalty data sets. Also, quite a few 
of the nonzero eigenvalues are rather small. Due to this, the 
numerical rank of the matrix is very low, a property that can
be exploited when estimating the entries of $e^{\cal A}$ using Lanczos-based
methods; see section \ref{sec:approx} for further discussion on this.

%\begin{table}[t!]
%\centering
%\caption{Top 10 hubs of the abortion
%web graph, ranked using $[e^{\mathcal{A}}]_{ii}$ and HITS.}
%\label{abortionhubs}
%\begin{tabular}{|c|c|c|c|}
%\hline
% $[e^{\mathcal{A}}]_{ii}$ nodes & $[e^{\mathcal{A}}]_{ii}$ score & 
%HITS nodes & HITS score  \\
% \hline
%48  & 3.3049e11 & 48  & .0.008968 \\
% \hline
% 1021 & 3.2080e11 & 1006  & .0.008836 \\
%\hline
% 1007 & 3.2080e11 & 1007  & .0.008836 \\
%\hline
%1006 & 3.2080e11 & 1021 & .0.008836 \\
%\hline
%1053 & 3.1958e11 &  1053 &0.008819 \\
%\hline
%1020 & 3.1956e11 & 1020 &0.008819 \\
%\hline
%987 & 3.1891e11 & 960 & 0.008810 \\
%\hline
%990 & 3.1891e11 & 968 &  0.008810  \\
%\hline
%985 & 3.1891e11 & 969 & 0.008810 \\
%\hline
%989 & 3.1891e11 & 970 & 0.008810 \\
%\hline
%\end{tabular}
%\end{table}

%\begin{table}[t!]
%\centering
%\caption{Top 10 authorities of the abortion
%web graph, ranked using $[e^{\mathcal{A}}]_{ii}$ and HITS.}
%\label{abortionauthorities}
%\begin{tabular}{|c|c|c|c|}
%\hline
% $[e^{\mathcal{A}}]_{ii}$ nodes & $[e^{\mathcal{A}}]_{ii}$ score & 
%HITS nodes & HITS score  \\
% \hline
% 967 & 4.0249e12 & 939  & 0.108667\\
% \hline
% 958 & 4.0249e12 & 958 & 0.108667 \\
%\hline
%939 & 4.02492e12 & 967 & 0.108667 \\
%\hline
%962 & 4.0012e12 & 961& 0.108347 \\
%\hline
%963 & 4.0012e12 & 962 & 0.108347 \\
%\hline
%964 &4.0012e12  & 963 & 0.108347\\
%\hline
%962 & 4.0012e12 & 964 & 0.108347 \\
%\hline
%965 & 4.0012e12 & 965 & 0.108347 \\
%\hline
%966 & 4.0012e12 & 966 & 0.108347 \\
%\hline
%587 & 1.3828e10 & 1582 & 0.002888 \\
%\hline
%\end{tabular}
%\end{table}

\begin{table}[t!]
\centering
\caption{Top 10 hubs of the abortion
web graph, ranked using $[e^{\mathcal{A}}]_{ii}$, HITS, Katz, 
$e^A$ row sums and Reverse PageRank with $\alpha = 0.85$.}
\label{abortionhubs}
\begin{tabular}{|c|c|c|c|c|}
\hline
 $[e^{\mathcal{A}}]_{ii}$ &
HITS & Katz & $e^A$ rs & RPR \\
 \hline
48  &  48  & 80  & 80 & 125  \\
 \hline
 1021 & 1006  & 1431  & 1431  & 2184  \\
\hline
 1007 & 1007  & 1432 & 1432  & 79  \\
\hline
1006 & 1021 & 1387 & 1426 & 81  \\
\hline
1053 &  1053 & 1388  & 1425 & 48  \\
\hline
1020 & 1020 & 1389  & 1415  & 1424  \\
\hline
987 & 960 & 1397  & 1388  & 1447  \\
\hline
990 & 968 & 1425  & 1389  &  78  \\
\hline
985 & 969 & 1426  & 1397  & 134  \\
\hline
989 & 970 & 1415  & 1387  & 1445  \\
\hline
\end{tabular}
\end{table}

\begin{table}[t!]
\centering
\caption{Top 10 authorities of the abortion
web graph, ranked using $[e^{\mathcal{A}}]_{ii}$, HITS, Katz,
$e^A$ column sums and PageRank with $\alpha=0.85$.}
\label{abortionauthorities}
\begin{tabular}{|c|c|c|c|c|}
\hline
 $[e^{\mathcal{A}}]_{ii}$ & 
HITS & Katz & $e^A$ cs  & PR  \\
 \hline
 967 & 939 & 1430  & 1430  & 1609 \\
 \hline
 958 & 958 & 1387  & 1387  & 1941 \\
\hline
 939 & 967 & 1425  & 1425  & 1948  \\
\hline
 962 & 961 & 1426  & 1426  & 1608 \\
\hline
 963 & 962 & 1429  & 1417  & 587 \\
\hline
 964 & 963 & 1396  & 1409  & 1610 \\
\hline
 961 & 964 & 1405 & 1429  & 2045  \\
\hline
 965 & 965 & 1406  & 1406 &  317  \\
\hline
966 & 966 & 1409  & 1396 & 2191  \\
\hline
587 & 1582 & 1417  & 1405  & 753  \\
\hline
\end{tabular}
\end{table}

The top 10 hubs and authorities of the abortion data set, as determined 
using the diagonal entries of
$e^{\mathcal{A}}$, HITS with constant initial vector, 
the row/column sums of $(I - cA)^{-1}$ (\lq\lq Katz\rq\rq), the 
row/column sums of $e^A$ and Reverse PageRank/PageRank are shown in
Tables \ref{abortionhubs} and \ref{abortionauthorities}. 
We observe that there is a good deal of agreement between the $e^{\cal A}$
rankings
and the HITS ones: indeed, both methods identify the websites 
labeled 48, 1021, 1007, 1006, 1053, 1020 as the top 6 hubs, and
both pick web site $48$ as the top one. Also, there are 7 
web sites identified by both methods as being among the top 10 
authorities. The top authority identified by HITS is ranked third
by $e^{\mathcal{A}}$, and conversely the top authority identified by 
$e^{\mathcal{A}}$ is third in the HITS ranking. 
%In contrast, the exponentiated input HITS algorithm returns 
%completely different rankings. 
%Node 80, which is deemed the best
%hub by the exponentiated input HITS, is ranked 1236 by HITS and
%851 using $e^{\mathcal{A}}$. 
%Exponentiated input HITS 
%ranks node 1430 as the top authority. This node is ranked 731
%by HITS and 429 by $e^{\mathcal{A}}$.
%The odd behavior of
%exponentiated input HITS is due to the fact that the 
%web graph abortion data set contains many nodes with no
%in-links (at least, none that are included in the data set).
%As seen in section \ref{example4}, nodes with no in-links are
%given less value by this algorithm and thus will not be
%reported as top hubs (nor will nodes pointed to by many
%nodes with no in-links be reported as top authorities).
%Since exponentiated input HITS behaved unreliably on the remaining 
%two data sets as well, we do not show
%the corresponding results.
The Katz rankings and those based on $e^A$ show considerable
agreement with one another, but are very different 
from the HITS ones and from those based on $e^{\cal A}$. 
Node 48, which is the top-ranked hub according to HITS
and $e^{\cal A}$, is now not even among the top 100. Conversely,
node 80, which is ranked the top hub by Katz and $e^A$, is not
in the top 100 nodes according to HITS or to $e^{\cal A}$.
This is not too surprising, since the metrics based on $A$ and those
based on $\cal A$ are obtained by counting 
rather different types of graph walks. Finally, for this
network Reverse PageRank
and PageRank return rankings with almost no overlap with any of
the other methods. 

\begin{figure}[t!]
\centering
\includegraphics[width=0.8\textwidth]{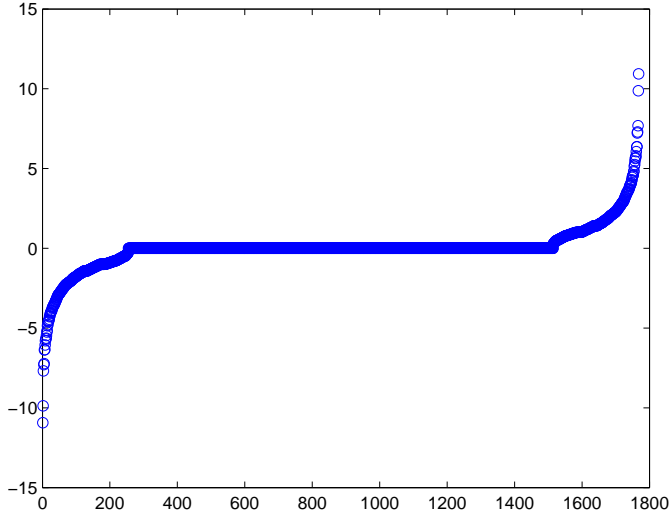}
\vspace{-0.25in}
\caption{Plot of the eigenvalues of the expanded computational 
complexity matrix $\mathcal{A}$. }
\label{compcomplexevalues}
\end{figure}

\subsection{Computational complexity data set}

The {\em computational complexity\/} data set contains 
$n=884$ nodes and $m= 1616$ directed edges.  
The expanded matrix $\mathcal{A}$ has order $N = 2n =1768$ and 
contains $2m=2232$ nonzeros.  The maximum eigenvalue of 
$\mathcal{A}$ is $\lambda_N \approx 10.93$ and the second 
largest eigenvalue is $\lambda_{N-1} \approx 9.86$.  Here, 
the (relative) spectral gap between the first and the second
eigenvalue is smaller than in the previous example; 
consequently, we expect the rankings produced 
using $e^{\mathcal{A}}$ and HITS to be less similar than for 
the abortion data set.  A plot of the eigenvalues of the 
expanded computational complexity data set matrix can be found 
in Fig.~\ref{compcomplexevalues}.

%\begin{table}[t!]
%\centering
%\caption{Top 10 hubs of the computational complexity 
%web graph, ranked using $[e^{\mathcal{A}}]_{ii}$ and HITS.}
%\label{comphubs}
%\begin{tabular}{|c|c|c|c|}
%\hline
% $[e^{\mathcal{A}}]_{ii}$ nodes & $[e^{\mathcal{A}}]_{ii}$ score & 
%HITS nodes & HITS score  \\
% \hline
% 57 & 2.6518e04 &  57 & 0.276589\\
% \hline
% 17 & 7.2059e02 &  634 & 0.035592 \\
%\hline
%644 & 6.6561e02 &  644 & 0.020557 \\
%\hline
%643 & 6.1256e02 &  721 & 0.018340\\
%\hline
%634 & 5.5558e02 &  643 & 0.017880\\
%\hline
%106 & 4.7486e02 &  554 & 0.014191\\
%\hline
%119 & 4.2791e02 &  632 & 0.013106 \\
%\hline
%529 & 3.8451e02 &  801 & 0.012383 \\
%\hline
%86 & 3.6528e02 & 640 & 0.011566\\
%\hline
%162 & 3.5502e02 & 639 & 0.010893\\
%\hline
%\end{tabular}
%\end{table}

%\begin{table}[h!]
%\centering
%\caption{Top 10 authorities of the computational complexity 
%web graph, ranked using $[e^{\mathcal{A}}]_{ii}$ and HITS.}
%\label{compauthorities}
%\begin{tabular}{|c|c|c|c|}
%\hline
% $[e^{\mathcal{A}}]_{ii}$ nodes & $[e^{\mathcal{A}}]_{ii}$  score &  
%HITS nodes & HITS score  \\
% \hline
% 1 & 4.8958e03 &  719 & 0.012155\\
% \hline
% 315 & 1.4747e03 &  717 & 0.011501 \\
%\hline
%673 & 8.0015e02 & 727 & 0.009972 \\
%\hline
%148 & 7.3093e02 & 723 & 0.009131\\
%\hline
%719 & 6.6746e02 & 808 & 0.008828\\
%\hline
%717 & 5.8437e02 & 735 & 0.008785\\
%\hline
%2 & 5.5637e02 &  737 & 0.008721 \\
%\hline
%45 & 4.0969e02 & 1 & 0.008550 \\
%\hline
%727 & 4.0315e02 & 722 & 0.008491\\
%\hline
%534 & 3.4473e02 & 770 & 0.008491\\
%\hline
%\end{tabular}
%\end{table}

The top 10 hubs and authorities of the computational complexity 
data set, determined by the various ranking methods, can be found in Tables 
\ref{comphubs} and \ref{compauthorities}.
%We also report the actual scores obtained for these nodes.
 As expected, we see
less agreement between HITS and the diagonals of $e^{\cal A}$. 
Concerning the
hubs, both methods agree that the web site labelled 57 is by far the most
important hub on the topic of computational complexity. However,
the method based on $e^{\cal A}$ identifies as the
second most important hub the web site corresponding to node 17, 
which is ranked only 39th by HITS. The two methods
agree on the next three hubs, but after that they return completely
different results. The difference is even more pronounced for the
authority rankings. The method based on $e^{\cal A}$ clearly
identifies web site 1 as the most authoritative one, whereas
HITS relegates this node to 8th place. The top authority
acording to HITS, web site 719, places 5th in the ranking 
obtained by $e^{\cal A}$. The two methods agree on only two
other web sites as being in the top 10 authorities (717 and 727). 
The Katz rankings and those based on $e^A$ show little overlap
for this data set, although node 57 is clearly considered an
important hub by all measures. A natural question is how much 
these results are affected by the choice of the parameter $c$
used to compute the Katz scores. We found experimentally that,
in contrast to the situation for the other data sets, small
changes in the value of $c$ can significantly affect the Katz
ranking for this particular data set. Changing the value of $c$
to $c = 1/(\rho(A) + 0.3)$ results in hub and authority rankings
that are much closer to those given by the column/row sums of
$e^A$. The potential sensitivity to $c$ is a clear drawback of the Katz-based
approach compared to the methods based on the matrix exponential.
Coming to (Reverse) PageRank, it is interesting to note that for
this data set it provides rankings that are at least in partial
agreement with some of the other measures, especially those
based on $e^A$. Looking at the authority scores, we also notice
a good degree of overlap among all methods, except HITS. Due to 
the small spectral gap, HITS is probably the least reliable of all
ranking methods on this particular data set.

\begin{table}[t!]
\centering
\caption{Top 10 hubs of the computational complexity
web graph, ranked using $[e^{\mathcal{A}}]_{ii}$, HITS, Katz, 
$e^A$ row sums and Reverse PageRank with $\alpha=0.85$.}
\label{comphubs}
\begin{tabular}{|c|c|c|c|c|}
\hline
 $[e^{\mathcal{A}}]_{ii}$ &
HITS & Katz & $e^A$ rs & RPR \\
 \hline
57  &  57  & 56  & 57  & 57  \\
 \hline
17 & 634  & 709  & 56 & 56  \\
\hline
644 & 644  & 57  & 17  & 17  \\
\hline
643 & 721 & 697  & 51  & 51  \\
\hline
634 & 643 & 705  & 634  & 21  \\
\hline
106 & 544 & 690  & 21  & 11  \\
\hline
119 & 632 & 714  & 255  & 255  \\
\hline
529 & 801 & 708   & 173  &  12  \\
\hline
86 & 640 & 712  & 709 & 13  \\
\hline
162 & 639 & 715  & 45 & 45  \\
\hline
\end{tabular}
\end{table}

\begin{table}%[t!]
\centering
\caption{Top 10 authorities of the computational complexity
web graph, ranked using $[e^{\mathcal{A}}]_{ii}$, HITS, Katz,
$e^A$ column sums and PageRank with $\alpha=0.85$.}
\label{compauthorities}
\begin{tabular}{|c|c|c|c|c|}
\hline
 $[e^{\mathcal{A}}]_{ii}$ & 
HITS & Katz & $e^A$ cs  & PR  \\
 \hline
 1 & 719 & 688 & 673  & 673 \\
 \hline
 315 & 717 & 685  & 1  & 664 \\
\hline
 673 & 727 & 673  & 664  & 534  \\
\hline
 148 & 723 & 690 & 534  & 45 \\
\hline
 719 & 808 & 56  & 45  & 2 \\
\hline
 717 & 735 & 686  & 473  & 1 \\
\hline
 2 & 737 & 664 & 315  & 376  \\
\hline
 45 & 1 & 1 & 376 &  341  \\
\hline
727 & 722 & 45  & 688  & 50  \\
\hline
534 & 770 & 534  & 599  & 51  \\
\hline
\end{tabular}
\end{table}

%\begin{figure}[ht]
%\centering
%\subfigure[Hub rankings]{
%\includegraphics[width=0.8\textwidth]{compcomplexityhubsemilogyrank.eps}
%}
%\subfigure[Authority rankings]{
%\includegraphics[width=0.8\textwidth]{compcomplexityauthsemilogyrank.eps}
%}
%\caption[Optional caption for list of figures]{A semilogy plot of the two normalized hub and authority rankings for each node in the computational complexity 
%data set. The rankings using $e^{\mathcal{A}}$ are in blue and the rankings using HITS are in red.}
%\label{compcomplexityplot}
%\end{figure}

\begin{figure}[t!]
\centering
\includegraphics[width=0.8\textwidth]{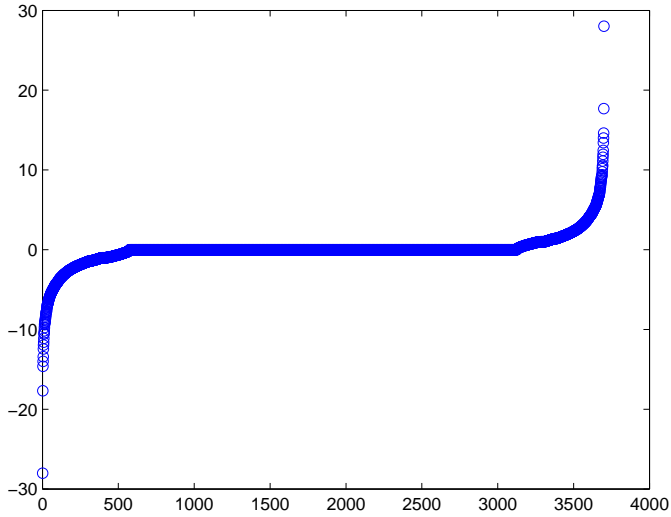}
\vspace{-0.25in}
\caption{Plot of the eigenvalues of the expanded death penalty matrix $\mathcal{A}$. }
\label{deathpenaltyevalues}
\end{figure}

\subsection{Death penalty data set}

The {\em death penalty} data set contains $n=1850$ and $m=7363$ 
directed edges.  The expanded matrix $\mathcal{A}$ has order $N=2n=3700$ 
and contains $m=14726$ nonzeros.  The maximum eigenvalue of 
$\mathcal{A}$ is $\lambda_N \approx 28.02$ and the second largest 
eigenvalue $\lambda_{N-1} \approx 17.68$.  In this case, the largest 
and second largest eigenvalues are quite far apart, and the relative 
gap is larger than in the previous examples.  
A plot of the eigenvalues of the
expanded death penalty matrix can be found in Fig.~\ref{deathpenaltyevalues}.

Due to the presence of a large spectral gap, 
much of the information used in forming the 
rankings of $e^{\mathcal{A}}$ is also used in the HITS ranking,
and we expect the two methods to produce similar results; see
section \ref{hits_comp}.
Indeed, as shown in Table \ref{deathhubs} (hubs) and
Table \ref{deathauthorities} (authorities),    
in this case the top 10 rankings produced by the
two methods are actually identical.

Looking at the Katz scores and those based on $e^A$, we see in
this case a great deal of overlap between these two, but almost
completely different rankings compared to HITS and $e^{\cal A}$ 
(although node 210 is clearly an important hub by any
standard). Note that node 1632 is both the top hub and the
top authority according to Katz and to $e^A$. PageRank and
Reverse PageRank show a limited amount of overlap with the other 
measures; nevertheless, nodes 210 and 1632 are also found 
to be important hubs and nodes 1632, 1 and 4 are found to be
authoritative, in agreeemnt with some of the other measures.

\subsection{Stanford web graph}
The {\em wb-cs-stanford} data set from the University of Florida sparse 
matrix collection contains $n=9914$ nodes and $m=36854$ directed edges.  
The expanded matrix $\mathcal{A}$ has order $N=2n=19828$ and contains $m=73708$ nonzeros.  The maximum
eigenvalue of $\mathcal{A}$ is $\lambda_N \approx 38.38$ and the 
second largest is $\lambda_{N-1} \approx 32.12$, hence there
is a sizeable gap.
Tables \ref{stanfordhubs}-\ref{stanfordauthorities} 
report the results obtained with the various ranking schemes.

\begin{table}[t!]
\centering
\caption{Top 10 hubs of the death penalty
web graph, ranked using $[e^{\mathcal{A}}]_{ii}$, HITS, Katz, 
$e^A$ row sums and Reverse PageRank with $\alpha=0.85$.}
\label{deathhubs}
\begin{tabular}{|c|c|c|c|c|}
\hline
 $[e^{\mathcal{A}}]_{ii}$ &
HITS & Katz & $e^A$ rs & RPR \\
 \hline
210  & 210  & 1632  & 1632  & 210  \\
 \hline
637 & 637  & 133  & 133  & 1632  \\
\hline
413 & 413  & 1671  & 1671 & 70  \\
\hline
1586 & 1586 & 552  & 552 & 95  \\
\hline
552 & 552 & 1651  & 1651  & 135  \\
\hline
462 & 462 & 1673  & 210 & 133  \\
\hline
930 & 930 & 1328  & 1673  & 55  \\
\hline
542 & 542 & 1653  & 1653  & 958  \\
\hline
618 & 618 & 210  & 1328  & 1077  \\
\hline
1275 & 1275  & 1709  & 1709 & 315  \\
\hline
\end{tabular}
\end{table}

\begin{table}%[t!]
\centering
\caption{Top 10 authorities of the death penalty
web graph, ranked using $[e^{\mathcal{A}}]_{ii}$, HITS, Katz,
$e^A$ column sums and PageRank with $\alpha=0.85$.}
\label{deathauthorities}
\begin{tabular}{|c|c|c|c|c|}
\hline
 $[e^{\mathcal{A}}]_{ii}$ & 
HITS & Katz & $e^A$ cs  & PR  \\
 \hline
 4 & 4 & 1632   & 1632  & 993 \\
 \hline
 1 & 1 & 1662 & 1662  & 667 \\
\hline
 6 & 6 & 1697  & 1697  & 3  \\
\hline
 7 & 7 & 1689  & 1689  & 736 \\
\hline
 10 & 10 & 1653  & 1653  & 735 \\
\hline
 16 & 16 & 1671   & 1671  & 1632 \\
\hline
 2 & 2 & 1675  & 1675 & 42  \\
\hline
 3 & 3 & 1684  & 1684 &  1  \\
\hline
44 & 44 & 798 & 789  & 4  \\
\hline
27 & 27 & 1652  & 1654  & 1212  \\
\hline
\end{tabular}
\end{table}

The first thing to observe is the remarkable agreement
between the HITS, $e^{\cal A}$, Katz, and $e^A$ rankings
of both hubs and authorities. This in stark contrast with
the results for the previous three data sets. Moreover,
many of the nodes that are ranked highly as hubs are also ranked
highly as authorities.  A plausible
explanation of these observations is that the adjacency matrix $A$
for this digraph is much closer to being symmetric than
in the other cases. Indeed, the percentage of \lq\lq bidirectional\rq\rq\
edges in the wb-cs-stanford graph is $47.63\%$; the
corresponding percentages for the abortion, computational
complexity and death penalty graphs are just $2.72\%$, $2.97\%$
and $4.02\%$, respectively. 

Interestingly, the (Reverse) PageRank results are now drastically
different fron the ones provides by all the other measures in nearly
all cases. 
The only (partial) exception is that PageRank finds nodes 6837, 6839 and 6840 
to be among the top 10 authorities; these three
nodes are identified as the three most authoritative ones by the
remaining methods.

\begin{table}[t!]
\centering
\caption{Top 10 hubs of the wb-cs-stanford
web graph, ranked using $[e^{\mathcal{A}}]_{ii}$, HITS, Katz, 
$e^A$ row sums and Reverse PageRank with $\alpha=0.85$.}
\label{stanfordhubs}
\begin{tabular}{|c|c|c|c|c|}
\hline
 $[e^{\mathcal{A}}]_{ii}$ &
HITS & Katz & $e^A$ rs & RPR \\
 \hline
6562  & 6562 & 6562 & 6562 & 251  \\
 \hline
6838 & 6838  & 6837 &  6837 & 252  \\
\hline
6840 & 6837  & 6838  & 6838 & 253  \\
\hline
6837 & 6839 & 6839  & 6839 & 254  \\
\hline
6839 & 6840 & 6840  & 6840 & 271  \\
\hline
6616 & 6616 & 6669  & 6669 & 2240  \\
\hline
6765 & 6615 & 6668  & 6668  & 2241  \\
\hline
6615 & 6765 & 6670  & 6670  & 2242  \\
\hline
6669 & 6669 & 6616  & 6616  & 2243   \\
\hline
6731 & 6731  & 6615  & 6615 & 348  \\
\hline
\end{tabular}
\end{table}

\begin{table}%[t!]
\centering
\caption{Top 10 authorities of the wb-cs-stanford
web graph, ranked using $[e^{\mathcal{A}}]_{ii}$, HITS, Katz,
$e^A$ column sums and PageRank with $\alpha=0.85$.}
\label{stanfordauthorities}
\begin{tabular}{|c|c|c|c|c|}
\hline
 $[e^{\mathcal{A}}]_{ii}$ & 
HITS & Katz & $e^A$ cs  & PR  \\
 \hline
6837 & 6837 & 6837   & 6837  & 2264 \\
 \hline
 6840 & 6839 & 6839  & 6839  & 8226 \\
\hline
 6839  & 6840 & 6840  & 6840  & 8059  \\
\hline
 6838 & 6838 & 6838  & 6838   & 8057 \\
\hline
 6617 & 6617 & 6573  & 6573  & 4485 \\
\hline
 6615 & 6615 & 6574  & 6575  & 5707 \\
\hline
 6766 & 6614  & 6575  & 6576 & 8225  \\
\hline
 6764 & 6616 & 6576  & 6577 & 6837  \\
\hline
 6616 & 6764 & 6577  & 6578  & 6839  \\
\hline
 6614 & 6766 & 6578  & 6579  & 6840  \\
\hline
\end{tabular}
\end{table}

\section{Approximating the matrix exponential}
\label{sec:approx}

Several approaches are available for computing the matrix exponential 
\cite{highambook}. A commonly used scheme is the one based on
Pad\'e approximation combined with the scaling and squaring
method \cite{higham,highambook}, implemented in Matlab by the
{\tt expm} function. For an $n\times n$ matrix,
this method requires $O(n^2)$ storage and $O(n^3)$ arithmetic
operations; any sparsity in $A$, if present, is not exploited in
currently available implementations. Evaluation of the matrix exponential
based on diagonalization also requires $O(n^2)$ storage and $O(n^3)$
operations. Furthermore, these methods cannot be easily adapted
to the case where only selected entries (e.g., the diagonal
ones) of the matrix exponential are of interest.

For the purpose of ranking hubs and authorities in a directed network,
only the main diagonal
of $e^{\cal A}$ is required. This can be done without having to
compute {\em all} the entries in $e^{\cal A}$. If some of the off-diagonal entries
(communicabilities) are desired, for example those between the highest
ranked hubs and/or authorities, it is also possible to compute them without
having to compute the whole matrix $e^{\cal A}$, which would be prohibitive
even for a moderately large network. We further emphasize that in most 
applications one is not so much interested in computing an exact ranking of
{\em all} the nodes in a digraph, but only in identifying the top $k$ ranked
nodes, where the integer $k$ is small compared to $n$ (for example,
$k=10$ or $k=20$). It is highly desirable to develop methods that are
capable of quickly identifying the top $k$ hubs/authorities without having
to compute accurate hub/authority scores for each node.

Efficient, accurate methods for estimating (or, in some cases, bounding) 
arbitrary entries in
a matrix function $f(A)$ have been developed by Golub, Meurant and collaborators
(see \cite{golubmeurantbook} and references therein) and first applied
to problems of network analysis by Benzi and Boito in \cite{BenziBoito};
see also \cite{bonchietal}.
Here we limit ourselves to a brief description of these methods, referring the reader
to \cite{BenziBoito} and \cite{golubmeurantbook}
 for further details. Let $A$ be a real, symmetric, $n\times n$ matrix
and let $f$ be a function defined on the spectrum of $A$.
 Consider the eigendecompositions
$A=Q \Lambda Q^T$ and $f(A)=Qf(\Lambda)Q^T$, where
$Q=[\phi_1, \ldots, \phi_n]$ and 
$\Lambda = {\rm diag}\,(\lambda_1,\ldots ,\lambda_n)$; here we assume
that the eigenvalues of $A$ are ordered as
$\lambda_1 \le \ldots \le \lambda_n$. For given vectors $u$ and $v$ we have
\begin{equation}
%u^T f(A) v=u^T Q f(\Lambda) Q^T v=w^Tf(\Lambda)z=\sum_{i=1}^n f(\lambda_i)w(i)z(i),
u^T f(A) v=u^T Q f(\Lambda) Q^T v=w^Tf(\Lambda)z=\sum_{k=1}^n f(\lambda_k)w_kz_k,
\end{equation}
where $w=Q^Tu=(w_k)$ and $z=Q^T v=(z_k)$.
In particular, for $f(A)=e^A$ we obtain
\begin{equation}
u^T e^A v= \sum_{k=1}^n e^{\lambda_k}w_kz_k.
\label{exp3}
\end{equation}
Choosing $u=v=e_i$ (the vector with the $i$th entry equal to 1 and all the
remaining ones equal to $0$) we obtain an expression for the subgraph
centrality of node $i$:
$$ SC(i) := \sum_{k=1}^n e^{\lambda_k}\phi_{k,i}^2\,,$$
where $\phi_{k,i}$ denotes the $i$th component of vector $\phi_k$.
Likewise, choosing $u=e_i$ and $v=e_j$ we obtain the following expression
for the communicability between node $i$ and node $j$:
$$ C(i,j) := \sum_{k=1}^n e^{\lambda_k}\phi_{k,i}\phi_{k,j}.$$
Analogous expressions hold for other matrix functions, such as the resolvent.

%Let now $e$ denote the vector with all entries equal to 1.
%The average communicability for node $p$ can be computed as
%$$\langle C(p)\rangle = \frac{1}{N-1}(e^T {\rm e}^A e_p - e_p^T{\rm e}^A e_p),$$
%showing that this quantity can be evaluated once the two bilinear
%forms $e^T {\rm e}^A e_p$ and $ e_p^T{\rm e}^A e_p$ have been computed.

Hence, the problem is reduced to evaluating bilinear
expressions of the form $u^T f(A) v$. 
Such bilinear forms can be thought of as Riemann- Stieltjes integrals
with respect to a (signed) spectral measure:
$$
u^T f(A) v=\int_a^bf(\lambda)d\mu(\lambda),\quad
\mu(\lambda)=\left\{\begin{array}{ll}
0,& \textnormal{ if }\lambda<a=\lambda_1,\\
\sum_{k=1}^iw_kz_k, & \textnormal{ if }\lambda_{i}\leq\lambda < \lambda_{i+1},\\
\sum_{k=1}^n w_kz_k, & \textnormal{ if }b=\lambda_n\leq\lambda.
\end{array}\right .
$$
This integral can be approximated by means of a Gauss-type quadrature rule:
\begin{equation}
\int_a^bf(\lambda)d\mu(\lambda)=\sum_{j=1}^pc_jf(t_j)+\sum_{k=1}^qv_kf(\tau_k)+R[f],
\label{quadrature}
\end{equation}
where $R[f]$ denotes the error. Here the nodes $\{t_j\}_{j=1}^p$ and the 
weights $\{c_j\}_{j=1}^p$ are unknown, whereas
the nodes $\{\tau_k\}_{k=1}^q$ are prescribed. We have

\begin{itemize}
\item $q=0$ for the Gauss rule,
\item $q=1$, $\tau_1=a$ or $\tau_1=b$ for the Gauss--Radau rule,
\item $q=2$, $\tau_1=a$ and $\tau_2=b$ for the Gauss--Lobatto rule.
\end{itemize}

For certain matrix functions, including the exponential and the resolvent,
these quadrature rules can be used to obtain lower and upper bounds on the
quantities of interest; prescribing additional 
quadrature nodes leads to tighter and tighter
bounds, which (in exact arithmetic) converge monotonically to the true values 
\cite{golubmeurantbook}.
The evaluation of these quadrature rules is mathematically equivalent to the
computation of orthogonal polynomials via a three-term recurrence, or,
equivalently, to the computation of entries and spectral information of a certain
tridiagonal matrix via the Lanczos
algorithm. Here we briefly recall how this can be done for the case of
the Gauss quadrature rule, when we wish to estimate
the $i$th diagonal entry of $f(A)$.
It follows from (\ref{quadrature}) that the quantity
of interest has the form $\sum_{j=1}^pc_jf(t_j)$.
This can be computed from the relation
(Theorem 3.4 in \cite{golubmeurantbook}):
$$
\sum_{j=1}^pc_jf(t_j)=e_1^Tf(J_p)e_1,
%\sum_{j=1}^nc_jf(t_j)=\langle e_1\,|\, f(J_n)\, |\, e_1\rangle,
$$
where
\begin{displaymath}
J_p=\left(\begin{array}{ccccc}
\omega_1 & \gamma_1 &&&\\
\gamma_1 & \omega_2 & \gamma_2 &&\\
& \ddots & \ddots & \ddots &\\
&& \gamma_{p-2} & \omega_{p-1} & \gamma_{p-1}\\
&&& \gamma_{p-1} & \omega_p \end{array}\right)
\end{displaymath}
is a tridiagonal matrix whose eigenvalues are the
Gauss nodes, whereas the Gauss weights are given
by the squares of the first
entries of the normalized eigenvectors of $J_p$.
The entries of $J_p$ are computed using the Lanczos algorithm
with starting vectors
$x_{-1}=0$ and $x_0=e_i$. 
Note that it is not required to compute all the
components of the eigenvectors of $J_p$ if one
uses the Golub--Welsch QR algorithm; 
see \cite{golubmeurantbook}.

For small $p$ (i.e., for a small number of Lanczos steps), computing the
$(1,1)$ entry of $f(J_p)$ is inexpensive.
The main cost in estimating one entry of $f(A)$ with this approach
is associated with the sparse matrix-vector multiplies in the Lanczos
algorithm applied to the
adjacency matrix $A$. If only a small, fixed number of iterations are performed
for each diagonal element of $f(A)$, as is usually the case,
the computational cost (per node) is
at most $O(n)$ for a sparse graph, resulting
in a total cost of $O(n^2)$ for computing the subgraph centrality
of every node in the network. If only $k< n$ subgraph centralities 
are wanted, with $k$ independent of $n$, then
the overall cost of the computation will be $O(n)$ provided   
that sparsity is carefully exploited in the Lanczos algorithm
and that only a small number $p$ of iterations (independent of $n$)
is carried out. Note, however, that depending on the connectivity
characteristics of the network under consideration, the prefactor
in the $O(n)$ estimate could be large. The algorithm can be
implemented so that the storage requirements are $O(n)$ for a
sparse network---that is, a network in which the total number of links
grows linearly in the number $n$ of nodes.

%As already mentioned,
%a nice feature of the approach based on Gauss quadrature is that it yields
%monotonically converging lower
%and upper bounds. As shown in section \ref{sec:tests} below, in a typical network 
%this allows for the rapid identification
%of nodes with high hub and authority scores, which often
%is all one needs. Indeed, as soon as the
%lower bound for a node becomes larger than the upper bound for another
%node, it is known that the former is ranked higher than the latter, and
%additional iterations cannot alter that fact. This is especially useful
%when we want to compare pairs of nodes in terms of their hub or 
%authority rankings.

When applying the approach based on Gauss quadrature rules
to the $2n\times 2n$ matrix $\cal A$, only matrix-vector
products with $A$ and its transpose are required, 
just like in the HITS algorithm. 
If only the hub scores are wanted, it is also possible to apply the
described techniques to the symmetric matrix $AA^T$ using the
function $f(\lambda) = \cosh(\sqrt{\lambda})$; the same applies
if only the authority scores are wanted, working this time with the matrix
$A^TA$. The problem with this approach is that only estimates (rather
than increasingly accurate lower and upper bound) can be obtained,
due to the fact that the function $f(\lambda) = \cosh(\sqrt{\lambda})$
is not strictly completely monotonic on the positive real axis. We refer
to \cite{BenziGolub} for details. In our experiments we always work
with the matrix $\cal A$, since we are interested in computing both
hub and authority scores.

\subsection{Test results}
\label{sec:tests}
Accurate evaluation of {\em all} 
the diagonal entries of $e^{\mathcal{A}}$ using quadrature
rules may be too expensive for
truly large-scale graphs. In most applications,
fortunately, it is not necessary to rank all the nodes in
the network: only the top few hubs and authorities are likely
to be of interest. 
When using quadrature rules, the number of quadrature nodes 
(Lanczos iterations) required to correctly rank 
the nodes as hubs or authorities varies and depends on 
both the eigenvalues of $e^{\mathcal{A}}$ and how close 
the diagonal entries are in value.  If the rankings of 
the nodes are
very close, it can take many iterations for the ordering 
to be exactly determined.  However, since estimates for diagonal 
entries are calculated individually, once the top 10 (say) 
nodes have been identified, additional iterations can be 
performed only on these nodes in order to determine their exact ranking.

\begin{table}[t!]
\centering
\caption{The number of iterations necessary for the top 10 
hubs or authorities to be determined (not necessarily in 
the correct order).}
%along with the time it took to calculate the bounds.}
\label{intop10}
\vspace{-0.2in}
\subtable{
\begin{tabular}{|c|c|c|}
\hline
 Dataset & hub (lower bound) & hub (upper bound)  \\
 \hline
Abortion & $>40$ & $>40$ \\
%Comp. Complex. & 3 (1.09s) & 3 (1.09s) \\
Comp. Complex. & 3  & 3  \\
%Death Penalty & 5 (6.10s) & 3 (3.28s) \\
Death Penalty & 5 & 3  \\
Stanford      & 8 & 8 \\
\hline
\end{tabular}
}
\subtable{
\vspace{1in}
\begin{tabular}{|c|c|c|}
\hline
 Dataset & authority (lower bound) & authority (upper bound)  \\
 \hline
%Abortion & 2 (1.57s) & 2 (1.57s)  \\
Abortion & 2  & 2  \\
%Comp. Complex. & 4 (1.36s) & 5 (2.23s)  \\
Comp. Complex. & 4 & 5   \\
%Death Penalty & 4 (3.71 s) & 2 (1.41s) \\
Death Penalty & 4 & 2  \\
Stanford      & 7 & 8 \\
\hline
\end{tabular}
}
\end{table}

Our approach exploits the monotonicity of the Gauss-Radau bounds:
as soon as the lower bound for node $i$ is above the upper bounds for
other nodes, we know that node $i$ will be ranked higher than those
othe nodes. This observation leads to a simple algorithm for
identifying the top-$k$ nodes.
The number of Lanczos iterations per node necessary %(and the time required) 
to identify the top $k=10$ hubs and authorities, using Gauss-Radau 
lower and upper bounds, for the four data sets from section 
\ref{sec:applications} is given in Table \ref{intop10}.  
Our implementation is based on Meurant's Matlab code \cite{mmq},
From the table it can be seen that, in most cases, only 
2-5 iterations per node are needed. An exception is the determination
of the top 10 
hubs of the abortion data set, for which the number 
of iterations is large ($>40$).  
This is due to a cluster of nodes (nodes 960 and 968-990) 
that have nearly identical hub rankings. These nodes' 
scores agree to 15 significant digits.
%the rankings range from 3.189134930290351e11 to 3.189134930290353e11.  
However, for  most applications, if
a subset of nodes are so closely ranked, their 
exact ordering may not be so important.  Table 
\ref{intop30} reports the number of Lanczos iterations neeeded for the 
top $k=10$ hubs and authorities to be ranked at least in the top 30.  
Here, the number of iterations per node needed is never more than 7.
The total cost is thus $O(n)$ Lanczos iterations, again leading to an
$O(n^2)$ overall complexity.
%Using a simple implementation based on Meurant's Matlab code \cite{mmq},
%which has not been optimized,
%the running times 
%to estimate the hub or authority score for all the nodes in the 
%data set using the number of Lanczos iterations listed range from 
%0.58s for the computational
%complexity data set
%to 6.41s for the abortion data set. 
Various enhancements can be used
to reduce costs,  
including the use of sparse-sparse mat-vecs
in the Lanczos iteration, and the
exclusion of nodes with zero out-degree (for hub computations)
and zero in-degree (for authority computations) from the 
top-$k$ calculations. It is also safe to assume that in most 
cases of interest, one can also
exclude nodes with in- and out-degree 1 from the computations,
leading to further savings.

\begin{table}[t!]
\centering
\caption{The number of iterations necessary for the top 10 hubs or 
authorities to be ranked in the top 30.}
%, along with the time it took 
%to calculate the bounds.}
\label{intop30}
\vspace{-0.2in}
\subtable{
\begin{tabular}{|c|c|c|}
\hline
 Dataset & hub (lower bound) & hub (upper bound)  \\
 \hline
%Abortion & 5 (8.54s) & 4 (6.41s) \\
Abortion & 5 & 4 \\
%Comp. Complex. & 2 (.58s) & 2 (.58s) \\
Comp. Complex. & 2 & 2  \\
%Death Penalty & 2 (1.75s) & 2 (1.75s) \\
Death Penalty & 2  & 2  \\
Stanford      & 7 & 4 \\
\hline
\end{tabular}
}
\subtable{
\begin{tabular}{|c|c|c|}
\hline
 Dataset & authority (lower bound) & authority (upper bound)  \\
 \hline
%Abortion & 2 (1.57s) & 2 (1.57s)  \\
Abortion & 2 & 2   \\
%Comp. Complex. & 4 (1.63s) & 2 (.59s)  \\
Comp. Complex. & 4 & 2   \\
%Death Penalty & 2 (1.41s) & 2(1.41s) \\
Death Penalty & 2  & 2 \\
Stanford      & 2 & 4 \\
\hline
\end{tabular}
}
\end{table}

\section{Conclusions and outlook}
In this paper we have presented a new approach to ranking hubs and
authorities in directed networks using functions of matrices. 
Bipartization is used to transform the original directed network
into an undirected one with twice the number of nodes. The adjacency
matrix of the bipartite graph is symmetric, and this allows the
use of subgraph centrality (and communicability) measures for undirected
networks. We showed that the diagonal entries of the matrix exponential  
provide hub and authority rankings, and we gave an interpretation
for the off-diagonal entries (communicabilities). Unlike HITS, the results are
independent of any starting vectors; and unlike the Katz-based ranking
schemes, there is no dependency on an arbitrary parameter.
% and the proposed method appears
%to be superior to the
%variant of HITS known as \lq\lq exponentiated 
%input HITS\rq\rq, which is unreliable in the presence of a large 
%number of nodes with 0 in-degree (a common occurrence in many real-life
%networks).

Several examples, both synthetic and corresponding to real data sets,
have been used to demonstrate the effectiveness of the proposed
ranking algorithms relative to HITS and to other ranking schemes
based on the matrix resolvent and on the exponential of the
adjacency matrix  of the original digraph. 
Our experiments indicate that our method
results in rankings that are frequently different from
those computed by HITS, at least in the absence of large gaps
between the dominant singular value of the adjacency matrix
$A$ and the remaining ones. This is to be expected,
since our method uses information from all the singular spectrum of the
network, not just the dominant left and right singular pairs.

As usual in this field, there is no simple way to compare different
ranking schemes, and therefore it is impossible to state with certainty
that a ranking scheme will give \lq\lq better\rq\rq\ results than a 
different scheme in practice.
 It is, however, certainly the case that the method
based on the exponential of $\cal A$ takes into account more spectral
information than HITS does; moreover, the rankings so obtained are
unambiguous, in that they do not depend on an the choice of an initial
guess or on a tuneable parameter. As we saw, the latter is a weak spot
of Katz-like methods, and a similar case can be made for 
PageRank and Reverse PageRank. 

Compared to HITS, the new technique has
a higher computational cost.
We showed how Gaussian quadrature rules
can be used to quickly identify the top ranked
hubs and authorities for networks involving
thousands of nodes. We note that such schemes
require a symmetric input matrix and are
not readily applicable to nonsymmetric matrices,
since in this case one can only hope for estimates 
instead of lower and upper bounds.  

Future work should include a more efficient implementation
and tests on larger networks. It is likely
that the proposed approach based on Gaussian quadrature will prove to be
too expensive for truly large-scale networks with millions of nodes. We
hope to explore techniques similar to those presented in \cite{bonchietal}
and \cite{dhillon} in order to extend our methodology to truly large-scale
networks. Another relevant question is the study of
the rate of convergence of the
Lanczos algorithm for estimating bilinear forms associated with adjacency
matrices of graphs of different types. %This work will be presented
%elsewhere. 
 
\section*{Acknowledgments}
We are indebted to Tammy Kolda (Sandia National Laboratories) 
and David Gleich (Purdue University) for valuable
suggestions.

\end{document}